\documentclass[12pt,reqno]{amsart}

\usepackage{amssymb}
\usepackage{fullpage}

\usepackage{algorithmic}
\usepackage{longtable}
\usepackage{caption}
\usepackage[makeroom]{cancel}

\usepackage{multirow}

\usepackage{tikz}
\usetikzlibrary{arrows,decorations.markings}

\theoremstyle{plain}
\newtheorem{theorem}{Theorem}[section]
\newtheorem{lemma}[theorem]{Lemma}

 \theoremstyle{remark}
\newtheorem{remark}[theorem]{Remark}

\theoremstyle{definition}
\newtheorem{definition}[theorem]{Definition}
\newtheorem{algorithm}[theorem]{Algorithm}

\numberwithin{equation}{section}

\def\sub{\subseteq}

\def\DynkinArrowLength{3mm}

\def\C{\mathbb C}
\def\F{\mathbb F}
\def\T{\mathbb T}
\def\Z{\mathbb Z}

\def\bB{\mathbf B}
\def\bG{\mathbf G}
\def\bH{\mathbf H}
\def\bT{\mathbf T}
\def\bU{\mathbf U}

\def\cA{\mathcal A}

\def\cD{\mathcal D}

\def\cI{\mathcal I}
\def\cJ{\mathcal J}
\def\cK{\mathcal K}
\def\cL{\mathcal L}

\def\cP{\mathcal P}
\def\cR{\mathcal R}
\def\cS{\mathcal S}
\def\cT{\mathcal T}

\def\cZ{\mathcal Z}

\def\rA{\mathrm A}
\def\rB{\mathrm B}
\def\rC{\mathrm C}
\def\rD{\mathrm D}
\def\rE{\mathrm E}
\def\rF{\mathrm F}
\def\rG{\mathrm G}
\def\rR{\mathrm R}
\def\rS{\mathrm S}
\def\rX{\mathrm X}

\def\fC{\mathfrak C}
\def\fO{\mathfrak O}
\def\fS{\mathfrak S}

\def\ua{{\underline a}}
\def\ub{{\underline b}}
\def\uc{{\underline c}}

\def\Ind{\operatorname{Ind}}
\def\Inf{\operatorname{Inf}}
\def\Irr{\operatorname{Irr}}
\def\Tr{\operatorname{Tr}}

\title{Constructing characters of Sylow $p$-subgroups of finite Chevalley groups}
\author{Simon M. Goodwin, Tung Le, Kay Magaard and Alessandro Paolini}

\address{School of Mathematics, University of Birmingham,
Birmingham, B15 2TT, U.K.} \email{s.m.goodwin@bham.ac.uk} \email{k.magaard@bham.ac.uk} \email{axp282@bham.ac.uk}

\address{Department of Mathematics, North-West University, Mafikeng 2735, South Africa}
\email{lttung96@yahoo.com}

\begin{document}

\begin{abstract}
Let $q$ be a power of a prime $p$, let $G$ be a finite Chevalley group over $\mathbb{F}_q$ and let $U$ be a
Sylow $p$-subgroup of $G$; we assume that $p$ is not a very bad prime for $G$.
We explain a procedure of reduction of irreducible complex characters of $U$, which leads to an algorithm
whose goal is to obtain a parametrization of the irreducible characters of $U$
along with a means to construct these characters as induced characters.
A focus in this paper is determining the parametrization when $G$ is of type $\rF_4$, where we observe
that the parametrization is ``uniform'' over good primes $p > 3$, but differs for the bad prime
$p = 3$.  We also explain how it has been applied for all groups of rank $4$ or less.
\end{abstract}

\maketitle

\section{Introduction}\label{sec:intr}
Let $q$ be a power of a prime $p$, and let $G$ be a finite Chevalley group over $\mathbb{F}_q$ and let $U$ be a
Sylow $p$-subgroup of $G$.  We assume that $p$ is not a very bad prime for $G$; recall that this means
that $p > 2$ if $G$ is of type $\rB_r$, $\rC_r$ or $\rF_4$, and $p > 3$ if $G$ is of type $\rG_2$.

We study the representation
theory of $U$ with the aim of determining a parametrization of the irreducible characters of $U$
and a means to construct them as induced characters of linear characters of certain subgroups.
Our principal tool for achieving this is a method of successively reducing characters to smaller subquotients
of $U$, which
leads to an algorithm whose goal is to determine the irreducible characters of $U$.
An outline of this algorithm is given below and explained more fully in Section \ref{sec:nonc}.

A focus of this paper is to obtain the parametrization in case $G$ is of type $\rF_4$, as stated in the following theorem.

\begin{theorem}
Let $q$ be a power of an odd prime $p$ and let $G$ be a finite Chevalley group over $\mathbb{F}_q$ of type $\rF_4$.
The irreducible characters of $U$ are completely parameterized in Table \ref{tab:f4}.  Moreover,
each character can be obtained as an induced character of a linear character of a certain subgroup
that can be determined from the information in Table \ref{tab:f4}.
\end{theorem}

As explained later in the introduction, the parametrization is ``uniform'' over all primes $p > 3$.
However, we observe significant
differences in the parametrization for the bad prime $p=3$.  These differences shed light on why the prime $p=3$ is bad for $G$ of type $\rF_4$.
In particular, we observe that for $p > 3$ all characters have degree $q^d$ for some $d \in \Z_{\ge 0}$, whereas for
$p = 3$ there are characters of degree $q^4/3$.  We note that similar behaviour for certain characters of $U$ when $G$
is of type $\rE_6$ (for $p =3$) or $\rE_8$ (for $p=5$) has previously been observed in \cite{LM2}.

We have also used our algorithm to determine a parametrization of the irreducible characters of $U$
for classical types up to rank 4.
We emphasise that our algorithm gives a construction of each character as an induced character
from a character of a certain subgroup of $U$, which gives a means to calculate the
values of these characters, see Theorem \ref{thm:1go}.  In fact for $G$ of rank 4 or less, we
obtain that each irreducible character of $U$ can be obtained by inducing a linear character.
In addition, we remark that our labelling of the irreducible characters is amenable to the action of
a maximal torus and also the field automorphisms; thus it would be straightforward to
determine these actions explicitly.

The methods in this paper
develop those used by Himstedt and the second and third authors in \cite{HLMd4} and \cite{HLMsr},
and make significant further progress.
A full parametrization  of the irreducible characters of $U$ for $G$ of type
$D_4$ and every prime $p$ is given in \cite{HLMd4}.
The so called single root minimal degree almost faithful irreducible characters
are parameterized for every type and rank when $p$ is not a very bad prime for $G$ in \cite{HLMsr}.

The approach used in those papers and this paper is built on partitioning the irreducible
characters of $U$ in terms of the root subgroups that lie in their centre, but not in their kernel.
Consequently, there are similarities to the theory of supercharacters, which were first
studied for the case $G$ is of type $\rA$ by Andr\'e, see for example \cite{An}.
This theory was fully developed by Diaconis and Isaacs in \cite{DI}.
Subsequently it was applied to the characters of $U$ for $G$ of types $\rB$, $\rC$ and $\rD$ by Andr\'e and
Neto in \cite{AN}.

Another approach to the character theory of $U$ is via the Kirillov orbit method, which
is applicable for $p$ greater than the Coxeter number of $G$.
In \cite{GMR2}, Mosch, R\"ohrle and the first author explain an algorithm for parameterizing the
coadjoint orbits of $U$, which was applied for $G$ of rank at most $8$, except $\mathrm{E}_8$; through the
Kirillov orbit method this leads to a parametrization of the irreducible characters of $U$.
This was preceded by an algorithm to determine the conjugacy classes
of $U$, see \cite{GMR1}.

We note that a reduction procedure for algebra groups similar to ours was
given by Evseev in \cite{Ev} and builds on work of Isaacs in \cite{Icount}.
For $G = \mathrm{SL}_n(q)$, this led to a parametrization of the irreducible characters of $U$ for $n \le 13$.
Also recently Pak and Soffer have determined the coadjoint orbits of $U$ for $G = \mathrm{SL}_n(q)$ and $n \le 16$, see \cite{PS}.
The situation for $G$ not of type $\rA$ turns out to be more complicated and we comment
more on this below.

There has been considerable other interest in the character theory and conjugacy
classes of $U$.  We refer the reader to \cite{LM1} or the introduction to \cite{GMR2}
for more information.

Motivation for this work lies in determining generic character tables for $U$, as has been
done for $G$ of type $D_4$ in \cite{GLM}.
This has a view towards applications
to the modular character theory of $G$ in nondefining characteristic; in particular, to determining
decomposition numbers;  see for example
\cite{Hi}, \cite{HH} and \cite{HN} for applications
of the character theory of parabolic subgroups to the modular representation theory
of $G$ in certain low rank cases.

We move on to give an outline of our algorithm to parameterize the irreducible characters of $U$;
we have omitted some details here and a full explanation is given in Section \ref{sec:nonc}.  In order to give this
outline we require some more notation.  We write $\Phi^+$ for the system of positive roots determined by $U$, and
for $\alpha \in \Phi^+$ we denote the corresponding root subgroup by $X_\alpha$.

In the algorithm we consider certain subquotients of $U$, which we refer
to as quattern groups.  A {\em pattern subgroup} of $U$ is
a subgroup that is a product of root subgroups, and a {\em quattern group} is a
quotient of a pattern subgroup by a normal pattern subgroup, we refer to
\S\ref{ss:quattern} for a precise definition.
A quattern group is determined by a subset $\cS$ of $\Phi^+$ and denoted by
$X_\cS$.  Given a subset $\cZ$ of $\{\alpha \in \cS \mid X_\alpha \sub Z(X_\cS)\}$, we
define
$$
\Irr(X_\cS)_\cZ = \{\chi \in \Irr(X_\cS) \mid X_\alpha \not\sub \ker \chi \text{ for all }
\alpha \in \cZ\}.
$$

At each stage of the algorithm, we are considering a pair $(\cS,\cZ)$ as above.
We attempt to apply one of two possible types of reductions
to reduce $(\cS,\cZ)$
to one or two pairs such that the
irreducible characters in $\Irr(X_\cS)_\cZ$ are in bijection
with those irreducible characters corresponding to the pairs we have
obtained in the reduction.

The first reduction is based on the elementary
but powerful character theoretic result \cite[Lemma 2.1]{HLMsr}, which is referred to as the reduction lemma.  In Lemma \ref{lem:small}, we state
and prove a specific
version of this lemma, which is the basis of the reduction.
This lemma shows that under certain conditions (which are straightforward to check) we can replace $(\cS,\cZ)$ with $(\cS',\cZ)$, where
$\cS'$ contains two fewer roots than $\cS$, and we have a
bijection between $\Irr(X_\cS)_\cZ$ and $\Irr(X_{\cS'})_\cZ$.

The second reduction is more elementary and used when it is not possible to apply the first reduction.  For this we choose
a root $\alpha$ such that $\alpha \not\in \cZ$, but $X_\alpha \sub Z(X_\cS)$.  Then
$(\cS,\cZ)$ is replaced with the two pairs $(\cS \setminus \{\alpha\},\cZ)$ and $(\cS,\cZ \cup \{\alpha\})$.  The
justification of this reduction is that $\Irr(X_\cS)_\cZ$ can be partitioned into the characters
in which $X_\alpha$ is contained in the kernel, namely $\Irr(X_{\cS \setminus \{\alpha\}})_\cZ$,
and the characters
in which $X_\alpha$ is not contained in the kernel, namely $\Irr(X_\cS)_{\cZ \cup \{\alpha\}}$.

We first partition the characters in terms of the root subgroups that
lie in their kernel, and then apply the reductions to each part of this partition.
After we have successively applied these reductions as many times as possible, we are
left with a set $\{(\cS_1,\cZ_1),\dots,(\cS_m,\cZ_m)\}$ for some $m \in \Z_{\ge 1}$ such that
$\Irr(U)$ is in bijection with the disjoint union
$$
\bigsqcup_{i=1}^m \Irr(X_{\cS_i})_{\cZ_i}.
$$
We refer to the pairs $(\cS_i,\cZ_i)$ as {\em cores}.  In many cases we have that $X_{\cS_i}$
is abelian in which case it is trivial to determine $\Irr(X_{\cS_i})_{\cZ_i}$.  The more interesting
cases are when $X_{\cS_i}$ is not abelian, we refer to these as {\em nonabelian cores},
where there is still some work required to determine
$\Irr(X_{\cS_i})_{\cZ_i}$.

As proved in Theorem \ref{thm:1go}, the irreducible characters of $U$ corresponding to $\Irr(X_{\cS_i})_{\cZ_i}$
are actually obtained from irreducible characters of $X_{\cS_i}$ by first inflating to a certain pattern
subgroup of $U$ and then inducing to $U$.  In particular, this gives a method to construct the
characters and, therefore, calculate the
values of these characters.

The algorithm has been implemented in the computer algebra system GAP3 \cite{GAP3}
using the CHEVIE package \cite{CHEVIE}.   For $G$ of rank $4$ or less, we have
used this and an analysis of the nonabelian cores obtained to determine a parametrization of $\Irr(U)$.
The results of the calculation are presented in the appendix for $G$ of types $\rB_4$, $\rC_4$ and $\rF_4$.

For the case where $G$ is of type $\rF_4$, we obtain six nonabelian cores.  These families of characters show the
most interesting behavior. For three of these families the parametrization of $\Irr(X_\cS)_\cZ$ is significantly
different when $p > 3$ and $p=3$; correspondingly, we get a different
expression for the size of $\Irr(X_\cS)_\cZ$ as a polynomial in $q$.
For $p = 3$, we obtain irreducible characters of degree $q^4/3$, whereas
for $p > 3$, we obtain that the degree of an irreducible character is always a power of $q$.

As mentioned above for $G$ of type $\rA$, a similar algorithm is given by Evseev in \cite{Ev}, which works in the framework of algebra groups.
This allows the algorithm to work with more general subgroups of $U$; there is not a natural analogue
of algebra groups in general types.  A parametrization of the irreducible characters of $U$ for $G$ or type $\rA$
up to rank $12$ is achieved in \cite{Ev}.  Indeed up to rank 11 there are no nonabelian cores (when working in the
framework of algebra groups).

The lack of an analogue of the more general notion
of algebra groups outside type $\rA$ leading to nonabelian cores in low rank
is in our opinion the main reason why the problem outside of type $\rA$ is more complex.
To deal with $G$ of higher rank, a more systematic procedure for dealing with nonabelian cores
is necessary, and is a direction for future research.  This should be based
on our analysis of nonabelian cores in Section \ref{sec:cors}.

Another direction for future work is
to construct generic character tables of $U$ for $G$ of type $\rF_4$, the case of $\rD_4$ given in \cite{GLM} serves as a model
of how to do this.

\medskip

\noindent
{\bf Acknowledgments:} We would like to thank Gunter Malle for a number of useful suggestions, and Frank Himstedt for some helpful comments. Also we are grateful to Eamonn O'Brien for verifying computationally the number of irreducible characters of $U$ of any fixed degree, when $G$ is of type $\rF_4$, and $q=3^e$ with $e \in \{1, 2, 3\}$.
Part of this research was completed during a visit of the second author to the University of Birmingham; we thank the LMS
for a grant to cover his expenses.

\section{Preliminaries}\label{sec:gens}

\subsection{Background on characters of finite groups.}
Let $G$ be a finite group, and let $H$ be a subgroup of $G$. We denote by $Z(G)$ the centre of $G$, and by $\Irr(G)$ the set of all irreducible
characters of $G$. We write $1_G$ for the trivial
character of $G$.  For a character $\eta \in \Irr(H)$, we write $\eta^G=\Ind_H^G \eta$ for the character of $G$ induced from $\eta$, and we denote
$$
\Irr(G \mid \eta)=\{\chi \in \Irr(G) \mid \langle \chi, \eta^G \rangle \ne 0\}.
$$
For a character $\chi \in \Irr(G)$, we denote
$$
\ker(\chi)=\{ g \in G \mid \chi(g)=\chi(1)\} \quad \text{and} \quad Z(\chi)=\{g \in G \mid |\chi(g)|=\chi(1)\}.
$$
Let $N$ be a normal subgroup of $G$.  We have an inflation map from $\Irr(G/N)$ to $\Irr(G)$ which takes $\chi \in \Irr(G/N)$ to
$\tilde \chi=\Inf_{G/N}^G \chi \in \Irr(G)$, where
$\tilde\chi(g)=\Inf_{G/N}^G \chi(g)=\chi(gN)$ for $g \in G$.
Given $g \in G$, $x \in N$ and $\psi \in \Irr(N)$, we write $x^g$ for $g^{-1}xg$
and we write ${}^g\psi : N \to \C$ for the
character defined by ${}^g\psi(x) = \psi(x^g)$.

For ease of reference later  we recall the following elementary commutativity
property of induction and inflation.  For $\psi \in \Irr(H/N)$ where $N \le H \le G$ and $N \unlhd G$, we have
\begin{equation} \label{eq:swap}
\Inf_{G/N}^G \Ind_{H/N}^{G/N} \psi= \Ind_H^G \Inf_{H/N}^H \psi.
\end{equation}

We next explain an elementary result, which we use in the sequel.
Let $Z$ and $T$ be subgroups of $Z(G)$ such that $Z\cap T=1$.
We can identify $Z$ with a subgroup of $G/T$. Let $\lambda \in \Irr(Z)$ and let
$\tilde \lambda$ denote its inflation to $ZT$.  Then it is straightforward
to show that we have a bijection
$\Irr(G \mid \tilde \lambda) \longleftrightarrow \Irr(G/T \mid \lambda)$.

The next lemma is key for our algorithm, it was proved in \cite[Lemma 2.1]{HLMsr} and
we refer to it as the reduction lemma.   We note that a similar result in the context of
algebra groups was previously proved by Evseev in \cite[Lemma 2.1]{Ev}.

\begin{lemma}[Reduction lemma] \label{lem:rl}
Let $G$ be a finite group, let $H \le G$ and let $X$ be a transversal of $H$ in $G$.
Let $Y$ and $Z$ be subgroups of $H$, and $\lambda \in \Irr(Z)$.  Suppose that
\begin{itemize} \item[(i)] $Z \sub Z(G)$,
\item[(ii)] $Y \trianglelefteq H$,
\item[(iii)] $Z \cap Y = 1$,
\item[(iv)] $ZY \trianglelefteq G$,
\item[(v)] for the inflation $\tilde \lambda \in \Irr(ZY)$ of $\lambda$, we have that
${^x}\tilde \lambda \ne {^y}\tilde \lambda$ for all $x,y \in X$ with $x \ne y$.
\end{itemize}
Then we have a bijection
\begin{align*}
\Irr(H/Y \mid \lambda) &  \to \Irr(G \mid \lambda) \cap \Irr(G \mid 1_Y)\\
\chi & \mapsto \Ind_H^G \Inf_{H/Y}^H \chi.
\end{align*}
Moreover, if $|X|=|Y|$, then $\Irr(G \mid \lambda) \cap \Irr(G \mid 1_Y) = \Irr(G \mid \lambda)$.
\end{lemma}

Let $p$ be a prime, and let $q=p^e$ for $e \in \Z_{\ge 1}$. We let $\F_q$ be the finite field with $q$ elements.
Denote by $\Tr : \F_q \to \F_p$ the trace map, and define $\phi : \F_q \to \C^\times$
by $\phi(x) = e^{\frac{i2\pi \Tr(x)}{p}}$, so that $\phi$ is a nontrivial character
from the additive group of $\F_q$ to the multiplicative group $\C^\times$.  We note
that $\ker \phi = \ker \Tr$.
For $a \in \F_q$, we define $\phi_a \in \Irr(\F_q)$ by $\phi_a(t) = \phi(at)$, and note that
$\Irr(\F_q)=\{\phi_a \mid a \in \F_q\}$.

It is clear that $\Tr(a_1s_1 + \dots + a_rs_r)=0$ for all
$s_1, \dots, s_r \in \F_q$ holds if and only if $a_1=\dots=a_r=0$.
Moreover, since the Frobenius automorphism $t \mapsto t^p$ is an automorphism of $\mathbb{F}_q$,
we have that the equality $\Tr(at^p)=0$ holds for all $t \in \mathbb{F}_q$ if and only if $a=0$.

The next lemma is important in our analysis of nonabelian cores; a version of this lemma giving $\ker \phi$ for an arbitrary choice
of character $\phi : \F_q \to \C^\times$, which is less explicit, was proved in \cite[Proposition 1.3]{LM1}.

\begin{lemma}\label{lem:TL} For a fixed $a \in \mathbb{F}_q^{\times}$, let $\mathbb{T}_a=\{t^p-a^{p-1}t \mid t \in \mathbb{F}_q\}$. Then
$$a^{-p} \, \mathbb{T}_a=\ker \Tr.$$
\end{lemma}

\begin{proof} We have that
\begin{align*}a^{-p} \, \mathbb{T}_a=\{a^{-p}(t^p-a^{p-1}t) \mid t \in \mathbb{F}_q\}=\{\left(ta^{-1} \right)^p-ta^{-1} \mid t \in \mathbb{F}_q\}=\{u^p - u \mid u \in \mathbb{F}_q\}.
\end{align*}
Now, we also have that
\begin{align*}\Tr(t^p-t)=\Tr(t^p)-\Tr(t)=\Tr(t)-\Tr(t)=0.
\end{align*}
Therefore,
$$\{t^p-t \mid t \in \mathbb{F}_q\} \sub \{x \in \mathbb{F}_q \mid \Tr(x)=0\} = \ker \Tr,$$
and all those sets have same cardinality $q/p$, therefore $ \ker(\Tr)=\{t^p-t \mid t \in \mathbb{F}_q\}=a^{-p} \, \mathbb{T}_a$.
\end{proof}

\subsection{Background on reductive groups} \label{ss:reductive}
We introduce now the main notation for finite reductive groups that we require.
We cite \cite[Section 3]{DM}, as a reference
for the theory of algebraic groups over finite fields, and for the
terminology used here.

Let $\bG$ be a connected reductive algebraic group defined and split over $\F_p$.
We assume that $p$ is not a very bad prime for $\bG$; recall that this means
that $p > 2$ if $\bG$ is of type $\rB_r$, $\rC_r$ or $\rF_4$, and $p > 3$ if $\bG$ is of type $\rG_2$.

Fix $\bB$ a Borel subgroup of $\bG$ defined over $\F_p$, and let $\bT$ be a maximal torus of $\bG$ contained in $\bB$
and defined over $\F_p$.
We write $\bU$ for the unipotent radical of $\bB$, which is defined over $\F_p$.
For a subgroup $\bH$ of $\bG$ defined over $\F_p$, we write $H = \bH(q)$
for the group of $\F_q$-rational points of $\bH$.  So $G = \bG(q)$
is a finite Chevalley group and $U = \bU(q)$ is a Sylow $p$-subgroup
of $G$.

For $G$ of type $\rX_r$, we sometimes write $U_{\rX_r}$ instead
of just $U$, so that we can discuss different groups at the same time\def\rS{\mathrm S}.

We denote by $\Phi$ the root system of $\bG$ with respect to $\bT$, and by $\Phi^+$ the set of positive roots in $\Phi$
determined by $\bB$. Let $N=|\Phi^+|$.
Recall that the standard (strict) partial order on $\Phi$ is defined by $\alpha < \beta$ if $\beta - \alpha$ is a sum of positive roots.
We fix an enumeration of $\Phi^+ = \{\alpha_1, \dots, \alpha_N \}$ such that $i < j $ whenever $\alpha_i < \alpha_j$.

For $\alpha \in \Phi^+$, we write $X_\alpha$ for the corresponding root subgroup of $U$
and we choose an isomorphism $x_\alpha : \F_q \to X_\alpha$.
We abbreviate and write $X_i$ for $X_{\alpha_i}$ and $x_i$ for $x_{\alpha_i}$.
Each element of $U$ can be written uniquely as $u=x_1(s_1)x_2(s_2) \cdots x_N(s_N),$ where $s_i \in \F_q$
for all $i=1, \dots, N$. In particular, the $X_i$ generate $U$, and $|U|=q^N$.

We now recall some standard facts about the commutator relations in $U$, we refer the
reader to \cite[Chapters 4 and 5]{Ca} for more details.
Given $\alpha,\beta \in \Phi^+$, we have
$$
[x_\alpha(r),x_\beta(s)] = \prod_{\substack{i, j > 0 : \\ i\alpha+j\beta \in \Phi^+ }} x_{i\alpha+j\beta}(c_{ij}^{\alpha,\beta}r^is^j)
$$
for certain coefficients $c_{ij}^{\alpha,\beta} \in \F_p$.  The parameterizations of the root subgroups can be chosen so that the coefficients
$c_{ij}^{\alpha,\beta}$  are always $\pm 1$, $\pm 2$, $\pm 3$, where $\pm 2$ occurs only for $G$ of type $\rB_r$, $\rC_r$, $\rF_4$ and $\rG_2$, and $\pm 3$
only occurs for $G$ of type $\rG_2$.
Moreover, as $p$ is not very bad for $G$, we have that
$$
[X_{\alpha}, X_{\beta}]=\prod_{\substack{i, j > 0 : \\ i\alpha+j\beta \in \Phi^+ }} X_{i\alpha+j\beta}
$$
for $\alpha, \beta \in \Phi^+$.

\subsection{Quattern groups} \label{ss:quattern}

In our algorithm for determining the irreducible characters of $U$, we require
certain subquotients of $U$, which we refer to as quattern groups.
The term pattern subgroup that we use below goes back to Isaacs, \cite[Section 3]{Icount}, and
quattern groups were also used in \cite{HLMsr}.
We give the required terminology and notation here.  Most of the
assertions made here are well-known, proofs can be found for example in \cite[Sections 3 and 4]{HLMsr}.

A subset $\cP$ of $\Phi^+$ is said to be {\em closed} (or a {\em pattern}) if for $\alpha, \beta \in \cP$,
we have that $\alpha+\beta \in \cP$ whenever $\alpha + \beta \in \Phi^+$.
For a closed subset $\cP$ of $\Phi^+$, we say that $\cK \sub \cP$ is {\em normal in $\cP$}, and write
$\cK \unlhd \cP$, if for all $\delta \in \cK$ and $\alpha \in \cP$,
we have $\delta+\alpha \in \cK$ whenever $\delta + \alpha \in \Phi^+$.
A subset $\cS$ of $\Phi^+$ is called a {\em quattern} if $\cS = \cP \setminus \cK$,
where $\cP$ is closed and $\cK$ is normal in $\cP$.

Let $\cP$ be a closed subset $\Phi^+$, let $\cK$ be normal in $\cP$, and let $\cS = \cP \setminus \cK$.
We define
$$
X_\cP = \prod_{\alpha \in \cP} X_\alpha.
$$
It is a straightforward exercise using the commutator relations to show that
$X_\cP$ is a subgroup of $U$.  We refer to a subgroup of
$U$ of the form $X_\cP$ as a {\em pattern group}.
Further, it is a consequence of the commutator relations
that $X_\cK$ is a normal subgroup of $X_\cP$, and we define
$$
X_\cS = X_{\cP \setminus \cK} = X_\cP/X_\cK.
$$
It follows from the construction of $X_\cS$ that the natural map
$\prod_{\alpha \in \cS} X_\alpha \to X_\cS$ is a bijection.

A subquotient
of $U$ of the form $X_\cS$ is called a {\em quattern} group.
We can easily check that $X_\cS$ is independent up to (canonical) isomorphism
of the possible choice of $\cP$ and $\cK$ such that $\cS = \cP \setminus \cK$, so there is
no ambiguity in the notation $X_\cS$.  We write $\cS = \cP \setminus \cK$ for a quattern,
where we are implicitly assuming that $\cP$ and $\cK$ are such a choice.
Given $\alpha \in \cS$, by a mild abuse of notation
we identify $X_\alpha$ with its image in $X_{\cS}$
for the remainder of this paper.

Let $\cS \sub \Phi^+$ be a quattern and let $X_\cS$ be the corresponding quattern group.
We define
$$
\cZ(\cS) =\{\gamma \in \cS \mid \gamma+\alpha \notin \cS \text{ for all } \alpha \in \cS\}
$$
and
$$
\cD(\cS) = \{\gamma \in \cZ(\cS) \mid \alpha+\beta \ne \gamma \text{ for all } \alpha, \beta \in \cS\}.
$$
Using the commutator relations and the assumption that $p$ is not very bad for $G$, it can be
shown that
$$
Z(X_\cS) = X_{\cZ(\cS)}.
$$
Then it can be seen that
$$
X_\cS \cong X_{\cS \setminus \cD(\cS)} \times X_{\cD(\cS)}.
$$

Let $\cS$ be a quattern and let $\cZ \sub \cZ(\cS)$.  We define
$$
\Irr(X_\cS)_\cZ = \{\chi \in \Irr(X_\cS) \mid X_\alpha \not\sub \ker(\chi) \text{ for all } \alpha \in \cZ\}.
$$
These sets of irreducible characters are key to the algorithm presented
in the next section.

Next we recall that a subset $\Sigma$ of $\Phi^+$ is called an {\em antichain} if for all
$\alpha, \beta \in \Sigma$, we have $\alpha \not< \beta$ and $\beta \not< \alpha$, i.e.\
$\alpha$ and $\beta$ are incomparable in the partial order on $\Phi^+$.

Given an antichain $\Sigma$ in $\Phi^+$, the set
$\cK_\Sigma=\{\beta \in \Phi^+ \mid \beta \nleq \gamma \text{ for all } \gamma \in \Sigma \}$ is a normal subset of $\Phi^+$.  Conversely,
given a normal subset $\cK$ of $\Phi^+$ the set
$\Sigma_\cK$ of maximal elements of $\Phi^+ \setminus \cK$ is clearly an antichain in $\Phi^+$.
This sets up a bijective correspondence between antichains in $\Phi^+$ and normal
subsets in $\Phi^+$.  The assertions made above are standard properties of posets,
see for example \cite[Section 4]{CP}.

For an antichain $\Sigma$ in $\Phi^+$, we define the quattern $\cS_\Sigma = \Phi^+ \setminus \cK_\Sigma$.
Then it is an easy consequence of the definitions that $\cZ(\cS_\Sigma) = \Sigma$.

Now let $\chi \in \Irr(U)$.  We define $\cR(\chi) = \{\alpha \in \Phi^+ \mid X_\alpha \sub \ker \chi\}$.
Using the commutator relations it is easy to see that $\cR(\chi)$
is a normal subset of $\Phi^+$, and thus $\Sigma_{\cR(\chi)}$ is an antichain in $\Phi^+$.
For an antichain $\Sigma \in \Phi^+$, we define $\Irr(U)_\Sigma = \{\chi \in \Irr(U) \mid \Sigma_{\cR(\chi)} = \Sigma\}$.
Then clearly we have the partition
$$
\Irr(U) = \bigsqcup_\Sigma \Irr(U)_\Sigma,
$$
where the union is taken over all antichains $\Sigma$ in $\Phi^+$.  Moreover,
we have that any character in $\Irr(U)_\Sigma$ is the inflation of an
irreducible character in $\Irr(X_{\cS_\Sigma})_\Sigma$.

We frequently want to inflate and induce characters from one quattern group to another, so we fix some
notation for this.  Let $\cS' = \cP' \setminus \cK'$ and $\cS = \cP \setminus \cK$ be quatterns, and let $\psi$
be a character of $X_{\cS'}$.  If $\cP' = \cP$ and $\cK' \supseteq \cK$, then we let $\cL = \cK' \setminus \cK$
and we write $\Inf_\cL \psi = \Inf_{X_{\cS'}}^{X_\cS} \psi$ for the inflation of $\psi$ from $X_{\cS'}$ to $X_\cS$; in case
$\cL = \{\alpha\}$ has one element, we write $\Inf_\alpha \psi = \Inf_\cL \psi$.  If
$\cK' = \cK$ and $\cP' \sub \cP$, then we let $\cT = \cP \setminus \cP'$
and we write $\Ind^{\cT} \psi$ for $\Ind_{X_{\cS'}}^{X_\cS}\psi$; in case $\cT = \{\alpha\}$ has one element,
we write $\Ind^\alpha \psi$ for $\Ind^\cT \psi$.

\subsection{Notation for $\mathrm{F}_4$.} \label{ss:f4}
We fix some specific notation in the case $\bG$ is of type $\rF_4$ that we use later.
In this case the Dynkin diagram of $\Phi$ is given in Figure \ref{fig:f4}, where $\Pi=\{\alpha_1, \alpha_2, \alpha_3, \alpha_4\}$ is
the set of simple roots determined by $\Phi^+$.
\begin{figure}[ht]
\begin{center}
\pgfdeclarelayer{background layer}
\pgfdeclarelayer{foreground layer}
\pgfsetlayers{background layer,main,foreground layer}
\begin{tikzpicture}[place/.style={circle,draw=black,fill=black, tiny},middlearrow/.style={
    decoration={markings,
      mark=at position 0.6 with
      {\draw (0:0mm) -- +(+135:\DynkinArrowLength); \draw (0:0mm) -- +(-135:\DynkinArrowLength);},
    },
    postaction={decorate}
  }, dedge/.style={
    middlearrow,
    double distance=0.5mm,
  }]
  \node(a) at (-3,0) [circle, draw, thick, fill=none, inner sep=2pt,label=below:$\alpha_1$] {};
  \node (b) at (-1,0) [circle, draw, thick, fill=none, inner sep=2pt,label=below:$\alpha_2$] {};
  \node (c) at (1,0) [circle, draw, thick, fill=none, inner sep=2pt,label=below:$\alpha_3$] {};
  \node (d) at (3,0) [circle, draw, thick, fill=none, inner sep=2pt,label=below:$\alpha_4$] {};
  \draw (a) -- (b);
  \draw (c) -- (d) ;
  \path (b) edge[dedge] (c);
\end{tikzpicture}
\caption{The Dynkin diagram of a root system of type $\mathrm{F}_4$.} \label{fig:f4}
\label{tab:F4}
\end{center}
\end{figure}

\noindent
There are $24$ positive roots in $\Phi$, listed in Table \ref{tab:prF4}; they are enumerated as in CHEVIE \cite{CHEVIE}.
We use the notation  \, {\footnotesize $\begin{matrix} 2 & 3 & 4 & 2     \end{matrix}$} \,  for the root $2\alpha_1+3\alpha_2+4\alpha_3+3\alpha_4$,
and similar notation for the other positive roots. The roots are enumerated so that their height is nondecreasing; we recall that the {\em height} of $\sum_{i=1}^4 a_i \alpha_i$ is by definition $\sum_{i=1}^4 a_i$.
We choose parameterizations of the root subgroups in $U$ so that the commutator relations are given as in Table \ref{tab:com}; all $[x_i(s),x_j(r)]$ not
listed in this table are equal to $1$.

\begin{center} \small
\begin{table}[ht]
\begin{tabular}{c|llll}
\hline
Height & Roots & & & \\
\hline
1 & $\alpha_{1}$ & $\alpha_{2}$ & $\alpha_{3}$ & $\alpha_{4}$ \\
\hline
2 & $\alpha_{5}=$ {\footnotesize $\begin{matrix} 1 & 1 & 0 & 0     \end{matrix}$} & $\alpha_{6}=$ {\footnotesize $ \begin{matrix} 0 & 1 & 1 & 0     \end{matrix}   $} & $\alpha_{7}=$ {\footnotesize $ \begin{matrix} 0 & 0 & 1 & 1     \end{matrix}   $} & \\
\hline
3 & $\alpha_{8}=$ {\footnotesize $ \begin{matrix} 1 & 1 & 1 & 0     \end{matrix}   $} & $\alpha_{9}=$ {\footnotesize $ \begin{matrix} 0 & 1 & 2 & 0     \end{matrix}   $} & $\alpha_{10}=$ {\footnotesize $ \begin{matrix} 0 & 1 & 1 & 1     \end{matrix}   $} & \\
\hline
4 & $\alpha_{11}=$ {\footnotesize $ \begin{matrix} 1 & 1 & 2 & 0     \end{matrix}   $} & $\alpha_{12}=$ {\footnotesize $ \begin{matrix} 1 & 1 & 1 & 1     \end{matrix}   $} & $\alpha_{13}=$ {\footnotesize $ \begin{matrix} 0 & 1 & 2 & 1     \end{matrix}   $} & \\
\hline
5 & $\alpha_{14}=$ {\footnotesize $ \begin{matrix} 1 & 2 & 2 & 0     \end{matrix}   $} & $\alpha_{15}=$ {\footnotesize $ \begin{matrix} 1 & 1 & 2 & 1     \end{matrix}   $} & $\alpha_{16}=$ {\footnotesize $ \begin{matrix} 0 & 1 & 2 & 2     \end{matrix} $} \,\,\,\, \,\,\,\, & \\
\hline
6 & $\alpha_{17}=$ {\footnotesize $ \begin{matrix} 1 & 2 & 2 & 1     \end{matrix}   $} & $\alpha_{18}=$ {\footnotesize $ \begin{matrix} 1 & 1 & 2 & 2     \end{matrix}   $} & & \\
\hline
7 & $\alpha_{19}=$ {\footnotesize $ \begin{matrix} 1 & 2 & 3 & 1     \end{matrix}   $} & $\alpha_{20}=$ {\footnotesize $ \begin{matrix} 1 & 2 & 2 & 2     \end{matrix} $} \,\,\,\, \,\,\,\, & & \\
\hline
8 & $\alpha_{21}=$ {\footnotesize $ \begin{matrix} 1 & 2 & 3 & 2     \end{matrix}   $} & & & \\
\hline
9 & $\alpha_{22}=$ {\footnotesize $ \begin{matrix} 1 & 2 & 4 & 2     \end{matrix}   $} & & & \\
\hline
10 & $\alpha_{23}=$ {\footnotesize $ \begin{matrix} 1 & 3 & 4 & 2     \end{matrix}   $} & & & \\
\hline
11 & $\alpha_{24}=$ {\footnotesize $ \begin{matrix} 2 & 3 & 4 & 2     \end{matrix} $} \,\,\,\,  \,\,\,\, & & & \\
\hline
\end{tabular}
\caption{Positive roots in a root system of type $\mathrm{F}_4$.}
\label{tab:prF4}
\end{table}
\end{center}

\begin{center} \small
\begin{table}[ht]
\begin{tabular}{ l l }
$
[x_{ 1 }(s),x_{ 2 }(r) ]=x_{ 5 }(rs)
$&$
[x_{ 1 }(s),x_{ 6 }(r) ]=x_{ 8 }(rs)
x_{ 14 }(-r^2s)
$\\$
[x_{ 1 }(s),x_{ 9 }(r) ]=x_{ 11 }(rs)
$&$
[x_{ 1 }(s),x_{ 10 }(r) ]=x_{ 12 }(rs)
x_{ 20 }(r^2s)
$\\$
[x_{ 1 }(s),x_{ 13 }(r) ]=x_{ 15 }(rs)
x_{ 22 }(r^2s)
$&$
[x_{ 1 }(s),x_{ 16 }(r) ]=x_{ 18 }(rs)
$\\$
[x_{ 1 }(s),x_{ 23 }(r) ]=x_{ 24 }(rs)
$&$

[x_{ 2 }(s),x_{ 3 }(r) ]=x_{ 6 }(rs)
x_{ 9 }(-r^2s)
$\\$
[x_{ 2 }(s),x_{ 7 }(r) ]=x_{ 10 }(rs)
x_{ 16 }(r^2s)
$&$
[x_{ 2 }(s),x_{ 11 }(r) ]=x_{ 14 }(rs)
$\\$
[x_{ 2 }(s),x_{ 15 }(r) ]=x_{ 17 }(rs)
x_{ 24 }(r^2s)
$&$
[x_{ 2 }(s),x_{ 18 }(r) ]=x_{ 20 }(rs)
$\\$
[x_{ 2 }(s),x_{ 22 }(r) ]=x_{ 23 }(rs)
$&$

[x_{ 3 }(s),x_{ 4 }(r) ]=x_{ 7 }(rs)
$\\$
[x_{ 3 }(s),x_{ 5 }(r)]=x_{ 8 }(-rs)
x_{ 11 }(rs^2)
$&$
[x_{ 3 }(s),x_{ 6 }( r )]=x_{ 9 }( 2 rs)
$\\$
[x_{ 3 }(s),x_{ 8 }( r )]=x_{ 11 }( 2 rs)
$&$
[x_{ 3 }(s),x_{ 10 }(r) ]=x_{ 13 }(rs)
$\\$
[x_{ 3 }(s),x_{ 12 }(r) ]=x_{ 15 }(rs)
$&$
[x_{ 3 }(s),x_{ 17 }(r) ]=x_{ 19 }(rs)
$\\$
[x_{ 3 }(s),x_{ 20 }(r) ]=x_{ 21 }(rs)
x_{ 22 }(-rs^2)
$&$
[x_{ 3 }(s),x_{ 21 }( r )]=x_{ 22 }( 2 rs)
$\\$

[x_{ 4 }(s),x_{ 6 }(r)]=x_{ 10 }(-rs)
$&$
[x_{ 4 }(s),x_{ 8 }(r)]=x_{ 12 }(-rs)
$\\$
[x_{ 4 }(s),x_{ 9 }(r)]=x_{ 13 }(-rs)
x_{ 16 }(rs^2)
$&$
[x_{ 4 }(s),x_{ 11 }(r)]=x_{ 15 }(-rs)
x_{ 18 }(rs^2)
$\\$
[x_{ 4 }(s),x_{ 13 }( r )]=x_{ 16 }( 2 rs)
$&$
[x_{ 4 }(s),x_{ 14 }(r)]=x_{ 17 }(-rs)
x_{ 20 }(rs^2)
$\\$
[x_{ 4 }(s),x_{ 15 }( r )]=x_{ 18 }( 2 rs)
$&$
[x_{ 4 }(s),x_{ 17 }( r )]=x_{ 20 }( 2 rs)
$\\$
[x_{ 4 }(s),x_{ 19 }(r) ]=x_{ 21 }(rs)
$&$

[x_{ 5 }(s),x_{ 7 }(r) ]=x_{ 12 }(rs)
x_{ 18 }(r^2s)
$\\$
[x_{ 5 }(s),x_{ 9 }(r)]=x_{ 14 }(-rs)
$&$
[x_{ 5 }(s),x_{ 13 }(r)]=x_{ 17 }(-rs)
x_{ 23 }(-r^2s)
$\\$
[x_{ 5 }(s),x_{ 16 }(r)]=x_{ 20 }(-rs)
$&$
[x_{ 5 }(s),x_{ 22 }(r) ]=x_{ 24 }(rs)
$\\$

[x_{ 6 }(s),x_{ 7 }(r)]=x_{ 13 }(-rs)
$&$
[x_{ 6 }(s),x_{ 8 }( r )]=x_{ 14 }( 2 rs)
$\\$
[x_{ 6 }(s),x_{ 12 }(r) ]=x_{ 17 }(rs)
$&$
[x_{ 6 }(s),x_{ 15 }(r)]=x_{ 19 }(-rs)
$\\$
[x_{ 6 }(s),x_{ 18 }(r)]=x_{ 21 }(-rs)
x_{ 23 }(rs^2)
$&$
[x_{ 6 }(s),x_{ 21 }( r )]=x_{ 23 }( 2 rs)
$\\$

[x_{ 7 }(s),x_{ 8 }(r) ]=x_{ 15 }(rs)
$&$
[x_{ 7 }(s),x_{ 10 }( r )]=x_{ 16 }( -2 rs)
$\\$
[x_{ 7 }(s),x_{ 12 }( r )]=x_{ 18 }( -2 rs)
$&$
[x_{ 7 }(s),x_{ 14 }(r)]=x_{ 19 }(-rs)
x_{ 22 }(rs^2)
$\\$
[x_{ 7 }(s),x_{ 17 }(r) ]=x_{ 21 }(rs)
$&$
[x_{ 7 }(s),x_{ 19 }( r )]=x_{ 22 }( 2 rs)
$\\$

[x_{ 8 }(s),x_{ 10 }(r)]=x_{ 17 }(-rs)
$&$
[x_{ 8 }(s),x_{ 13 }(r) ]=x_{ 19 }(rs)
$\\$
[x_{ 8 }(s),x_{ 16 }(r) ]=x_{ 21 }(rs)
x_{ 24 }(-rs^2)
$&$
[x_{ 8 }(s),x_{ 21 }( r )]=x_{ 24 }( 2 rs)
$\\$

[x_{ 9 }(s),x_{ 12 }(r) ]=x_{ 19 }(rs)
x_{ 24 }(-r^2s)
$&$
[x_{ 9 }(s),x_{ 18 }(r)]=x_{ 22 }(-rs)
$\\$
[x_{ 9 }(s),x_{ 20 }(r)]=x_{ 23 }(-rs)
$&$

[x_{ 10 }(s),x_{ 11 }(r) ]=x_{ 19 }(rs)
x_{ 23 }(-rs^2)
$\\$
[x_{ 10 }(s),x_{ 12 }( r )]=x_{ 20 }( -2 rs)
$&$
[x_{ 10 }(s),x_{ 15 }(r)]=x_{ 21 }(-rs)
$\\$
[x_{ 10 }(s),x_{ 19 }( r )]=x_{ 23 }( 2 rs)
$&$

[x_{ 11 }(s),x_{ 16 }(r) ]=x_{ 22 }(rs)
$\\$
[x_{ 11 }(s),x_{ 20 }(r)]=x_{ 24 }(-rs)
$&$

[x_{ 12 }(s),x_{ 13 }(r) ]=x_{ 21 }(rs)
$\\$
[x_{ 12 }(s),x_{ 19 }( r )]=x_{ 24 }( 2 rs)
$&$

[x_{ 13 }(s),x_{ 15 }( r )]=x_{ 22 }( -2 rs)
$\\$
[x_{ 13 }(s),x_{ 17 }( r )]=x_{ 23 }( -2 rs)
$&$

[x_{ 14 }(s),x_{ 16 }(r) ]=x_{ 23 }(rs)
$\\$
[x_{ 14 }(s),x_{ 18 }(r) ]=x_{ 24 }(rs)
$&$

[x_{ 15 }(s),x_{ 17 }( r )]=x_{ 24 }( -2 rs)
$\\

\end{tabular}
\caption{Commutator relations for $U$ for $G$ of type $\mathrm{F}_4$.}
\label{tab:com}
\end{table}
\end{center}

\section{Algorithm to parameterize the irreducible characters of $U$} \label{sec:nonc}

\subsection{Lemmas required for the algorithm}
Before describing our algorithm to determine the irreducible characters
of $U$, we present a couple of lemmas which are the basis
of the reductions performed in the algorithm.

Our first lemma gives a specific version of Lemma \ref{lem:rl}.

\begin{lemma}\label{lem:small} Let $\cS=\cP \setminus \cK$ be a quattern, let $\cZ \sub \cZ(\cS)$ and let $\gamma \in \cZ$.
Suppose that there exist $\delta, \beta \in \cS \setminus \{\gamma\}$, with $\delta+\beta=\gamma$, such that for all $\alpha, \alpha' \in \cS$
we have $\alpha + \alpha' \ne \beta$,
and that for all $\alpha \in \cS \setminus \{\beta\}$ we have
$\delta+\alpha \not\in \cS$.
Let $\cP'=\cP \setminus \{\beta\}$ and $\cK'=\cK \cup \{\delta \}$.  Then we have that $\cS' = \cP' \setminus \cK'$ is a quattern with $X_{\cS'} \cong X_{\cP'}/X_{\cK'}$, and
we have a bijection
\begin{align*}
\Irr(X_{\cS'})_\cZ &  \to \Irr(X_\cS)_\cZ \\
\chi & \mapsto \Ind^\beta \Inf_\delta \chi
\end{align*}
by inflating over $X_{\delta}$ and inducing to $X_\cS$ over $X_{\beta}$.
\end{lemma}

\begin{proof}
Let $\alpha, \alpha' \in \cP'$. If $\alpha \in \cK$ or $\alpha' \in \cK$, then it cannot be that $\alpha+\alpha'=\beta$, since in that case we would get $\beta \in \cK$, a contradiction with $\beta \in \cS$. If $\alpha, \alpha' \in \cS'$, by assumption the equality $\alpha+\alpha'=\beta$ cannot hold as well. Since $\cP'=\cS' \cup \cK$, this proves that $\cP'$ is closed.

Let now $\alpha \in \cP'$, and $\alpha' \in \cK'$. If $\alpha' \in \cK$, then $\alpha+\alpha' \in \cK'$ whenever $\alpha + \alpha' \in \Phi^+$, since $\cK \trianglelefteq \cP$. Otherwise, $\alpha'=\delta$, and by assumption $\alpha+\delta \notin \cS$ since $\alpha \neq \beta$, therefore if $\alpha + \delta \in \Phi^+$ then $\alpha+\delta \in \cK'$. Therefore $\cK' \trianglelefteq \cP'$, and $\cS'=\cP'\setminus \cK'$ is a quattern.

It is immediate that conditions (i)--(iv) of Lemma \ref{lem:rl} hold with
$G = X_\cS$, $Z=X_{\gamma}$, $H = X_{\cS \setminus \{\beta\}}$, $X = X_\beta$ and $Y = X_\delta$.

Let
$\lambda \in \Irr(Z)$ and $\tilde \lambda = \Inf_\delta \lambda$.   Then for $s_1, s_2 \in \F_q$, we have
$$
 {}^{x_{\beta}(s_1)} \tilde\lambda={}^{x_{\beta}(s_2)} \tilde \lambda \text{ if and only if } \lambda([x_{\beta}(s_1), x_{\delta}(t)])=\lambda([x_{\beta}(s_2), x_{\delta}(t)]) \text{ for all } t \in \F_q.
$$
Therefore, the commutator formulas in \S\ref{ss:reductive} imply that condition (v) must be satisfied, and of course $|X| = |Y| = q$, so the lemma follows.
\end{proof}

Our second lemma is an immediate consequence of the definitions.  We state
it for ease of reference later, and omit any proof.

\begin{lemma} \label{lem:split}
Let $\cS$ be a quattern and let $\alpha \in \cZ(\cS)$.  Then there is a bijection
$\Irr(X_\cS) \to \Irr(X_\cS)_{\{\alpha\}} \sqcup \Irr(X_{\cS \setminus \{\alpha\}})$.
\end{lemma}

\subsection{An example of the algorithm}\label{sub:ex}
Before we give a description of our algorithm, we illustrate it in an example.
We consider a case for $G$ of type $\rF_4$ and  use the notation given in \S\ref{ss:f4}.

We want to compute $\Irr(U)_{\Sigma}$, where $\Sigma=\{\alpha_{12}\}$.
We let $\cS = \cS_\Sigma =
\Phi^+ \setminus \cK_{\Sigma}$, so $\cS = \{\alpha_1,\dots,\alpha_8\} \cup \{\alpha_{10},\alpha_{12}\}$.  Also we let $\cZ = \Sigma = \{\alpha_{12}\}$.
So we want to compute $\Irr(X_\cS)_\cZ$.

Let
$$
(\beta_1,\delta_1)=(\alpha_1,\alpha_{10}), \qquad (\beta_2,\delta_2)=(\alpha_4,\alpha_8), \qquad (\beta_3,\delta_3)=(\alpha_5,\alpha_7).
$$
An application of Lemma \ref{lem:small}, for $(\beta,\delta) = (\beta_1,\delta_1)$ gives a bijection
$\Irr(X_{\cS^1})_\cZ \to \Irr(X_{\cS})_\cZ$,
where $\cS^1 = \cS \setminus \{\beta_1,\delta_1\}$.  Two further applications
give bijections $\Irr(X_{\cS^2})_\cZ \to \Irr(X_{\cS^1})_\cZ$ and $\Irr(X_{\cS^3})_\cZ \to \Irr(X_{\cS^2})_\cZ$,
where $\cS^2 = \cS^1 \setminus \{\beta_2,\delta_2\}$ and $\cS^3 = \cS^2 \setminus \{\beta_3,\delta_3\}$.
We record the sets $\cA = \{\beta_1,\beta_2,\beta_3\}$ and $\cL = \{\delta_1,\delta_2,\delta_3\}$
to remind us which reductions were performed. We also define $\cK=\cK_{\Sigma} \cup \cL$. These three reductions are all instances
of TYPE R reductions (the capitalized R means ``reduction lemma") in Algorithm \ref{alg:main} in the next subsection.

Now we can see that $\alpha_{12} \in \cD(\cS^3)$, so that
$X_{\cS^3} \cong X_{\cS^3 \setminus \{\alpha_{12}\}} \times X_{12}$.
In particular, this means there is no possibility to apply
Lemma \ref{lem:small}, with $\gamma \in \cZ = \{\alpha_{12}\}$.

We find that $\cZ(\cS^3) \setminus \cD(\cS^3) = \{\alpha_6\}$.
We can apply Lemma \ref{lem:split} to obtain a bijection
$$
\Irr(X_{\cS^3})_\cZ \to \Irr(X_{\cS^3})_{\cZ \cup\{\alpha_6\}} \sqcup \Irr(X_{\cS^3 \setminus \{\alpha_6\}}).
$$
We now split the two cases and consider them in turn.  We note that this
is an example of a TYPE S reduction (the capitalized S means ``split") as defined in our algorithm in the next subsection.

First we consider
$$
\Irr(X_{\cS^3})_{\cZ^3},
$$
where $\cS^3 = \{\alpha_2, \alpha_3, \alpha_6, \alpha_{12}\}$ and $\cZ^3 = \{\alpha_6,\alpha_{12}\}$.
We can apply Lemma \ref{lem:small} with $\delta=\alpha_3$, $\beta=\alpha_2$, and $\gamma=\alpha_6$. We then get a bijection
$\Irr(X_{\cS^4})_{\cZ^3} \longrightarrow \Irr(X_{\cS^3})_{\cZ^3}$, where $\cS^4 = \cS^3 \setminus \{\alpha_2,\alpha_3\} = \{\alpha_6,\alpha_{12}\}$.
This is another reduction of TYPE R as defined in the next subsection.  We record this reduction
by adjoining $\alpha_2$ to $\cA$ to obtain $\cA' = \{\alpha_1,\alpha_4,\alpha_5, \alpha_2\}$
and adjoining $\alpha_3$ to $\cL$ to obtain $\cL' = \{\alpha_{10},\alpha_8,\alpha_7, \alpha_3\}$. Moreover, we put $\cK'=\cK_{\Sigma} \cup \cL'$.

We note that $X_{\cS^4} = X_6 \times X_{12}$, so we can parameterize
$\Irr(X_{\cS^4})_{\cZ^3}$ as $\{\lambda^{a_6,a_{12}} \mid a_6,a_{12} \in \F_q^\times\}$, where
$\lambda^{a_6,a_{12}}(x_6(t)) = \phi(a_6t)$ and $\lambda^{a_6,a_{12}}(x_{12}(t)) = \phi(a_{12}t)$.
Through the bijections given by Lemma \ref{lem:small}, we obtain characters
of $U$ forming part of $\Irr(U)_\Sigma$ by a process of successive inflation and induction of
the characters $\lambda^{a_6,a_{12}}$.  These characters are
$$
\chi^{a_6,a_{12}} = \Inf_{\cK_\Sigma}   \Ind^{\alpha_1} \Inf_{\alpha_{10}} \Ind^{\alpha_4} \Inf_{\alpha_8}\Ind^{\alpha_5} \Inf_{\alpha_7}\Ind^{\alpha_2} \Inf_{\alpha_3}\lambda^{a_6,a_{12}}.
$$
However, it turns out that these characters can be obtained by a single inflation and then induction,
thanks to Theorem \ref{thm:1go}, and we have
$$
\chi^{a_6,a_{12}} = \Ind^{\cA'} \Inf_{\cK'} \lambda^{a_6,a_{12}}.
$$
The characters $\chi^{a_6,a_{12}}$ have degree $q^4$.

Next we move on to consider the characters in $\Irr(X_{\cS^5})_\cZ$
where $\cS^5 = \cS^3 \setminus \{\alpha_6\} = \{\alpha_2,\alpha_3,\alpha_{12}\}$, and $\cZ = \{\alpha_{12}\}$.
We record that we have put $\alpha_6$ in the kernel by adjoining it to $\cK$ to
obtain $\cK'' = \cK \cup \{\alpha_6\}$.
We see that $X_{\cS^5}$ is abelian, so that
$\Irr(X_{\cS^5}) = \{\lambda_{b_2, b_3}^{a_{12}} \mid a_{12} \in \F_q^\times, b_2,b_3 \in \F_q\}$, where
$\lambda_{b_2, b_3}^{a_{12}}(x_2(t))=\phi(b_{2}t)$ and $\lambda_{b_2, b_3}^{a_{12}}(x_3(t))=\phi(b_{3}t)$,
and $\lambda_{b_2, b_3}^{a_{12}}(x_{12}(t)) = \phi(a_{12}t)$.

Now through the bijections previously obtained in Lemma \ref{lem:small}, we obtain characters $\chi_{b_2,b_3}^{a_{12}}$
of $U$ forming part of $\Irr(U)_\Sigma$ from the characters $\lambda_{b_2,b_3}^{a_{12}}$
by a process of successive inflation and induction.
We have
$$
\chi_{b_2,b_3}^{a_{12}} = \Inf_{\cK_\Sigma}  \Ind^{\alpha_1} \Inf_{\alpha_{10}}\Ind^{\alpha_4} \Inf_{\alpha_8}  \Ind^{\alpha_5} \Inf_{\alpha_7}\Inf_{\alpha_6}\lambda_{b_2,b_3}^{a_{12}},
$$
and note that by using Theorem \ref{thm:1go}, we
can write these characters as
$$
\chi_{b_2,b_3}^{a_{12}} = \Ind^{\cA} \Inf_{\cK''} \lambda^{a_6,a_{12}}.
$$
These characters have degree $q^3$.

Putting this together, we have that
$$
\Irr(U)_{\{\alpha_{12}\}}=\{\chi^{a_6, a_{12}} \mid a_6, a_{12} \in \mathbb{F}_q^{\times}\}
\sqcup \{ \chi_{b_2, b_3}^{a_{12}} \mid  b_2, b_3 \in \mathbb{F}_q, a_{12} \in \mathbb{F}_q^{\times}\}.
$$
Therefore, $\Irr(U)_{\{\alpha_{12}\}}$ consists of:
\begin{itemize}
\item $(q-1)^2$ characters of degree $q^4$; and
\item $q^2(q-1)$ characters of degree $q^3$.
\end{itemize}

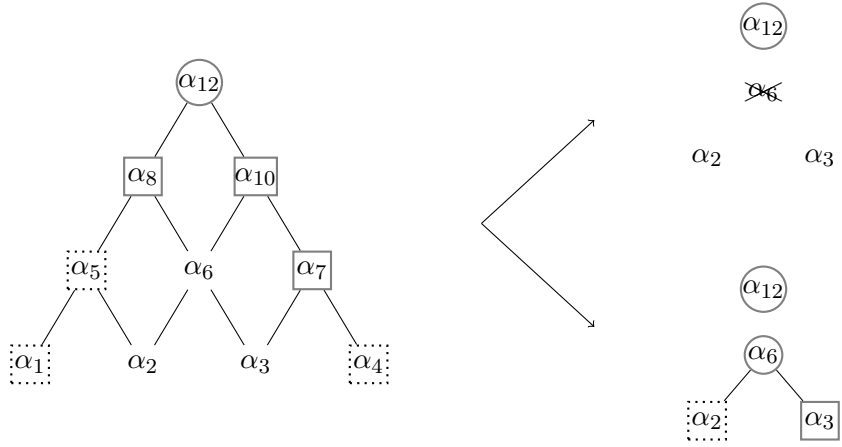
\begin{figure}
\begin{center}

\begin{tikzpicture}[scale=0.5, place/.style={circle,draw=black!50,fill=blue!20,thick,
inner sep=0pt,minimum size=6mm}, transition/.style={rectangle,draw=black!50,thick,
inner sep=0pt,minimum size=5mm},td1/.style={rectangle,draw=black,ultra thick, dotted,
inner sep=0pt,minimum size=5mm},td2/.style={rectangle,draw=black,thick, dotted,
inner sep=0pt,minimum size=5mm},td3/.style={rectangle,draw=black,thin, dotted,
inner sep=0pt,minimum size=5mm},  rrt/.style={circle,draw=black!50,thick,
inner sep=0pt,minimum size=5mm}, brt1/.style={rectangle,draw=black!50,ultra thick,
inner sep=0pt,minimum size=5mm}, brt2/.style={rectangle,draw=black!50,thick,
inner sep=0pt,minimum size=5mm}, brt3/.style={rectangle,draw=black!50,thin,
inner sep=0pt,minimum size=5mm},
grt/.style={rectangle,draw=gray!50,thick,
inner sep=0pt,minimum size=5mm},
crc/.style={circle,draw=black!50,thick,
inner sep=0pt,minimum size=5mm}, terminal/.style={
rectangle,minimum size=6mm,rounded corners=3mm,
very thick,draw=black!50, top color=white,bottom color=black!20, font=\ttfamily}]

\node (a) at (-4.5,0) [td2] {\small{$\alpha_1$}};
  \node (b) at (-1.5,0)  {\small{$\alpha_2$}};
  \node (c) at (1.5,0)  {\small{$\alpha_3$}};
  \node (d) at (4.5,0) [td2] {\small{$\alpha_4$}};
  \node (e) at (-3,2.5) [td2] {\small{$\alpha_5$}};
  \node (f) at (0,2.5)  {\small{$\alpha_6$}};
    \node (g) at (3,2.5) [brt2] {\small{$\alpha_7$}};

  \node (h) at (-1.5, 5) [brt2] {\small{$\alpha_8$}};

  \node (j) at (1.5,5) [brt2] {\small{$\alpha_{10}$}};

  \node (l) at (0,7.5) [rrt] {\small{$\alpha_{12}$}};

    \node (m) at (13.5,5.5)  {\small{$\alpha_2$}};

    \node (n) at (16.5,5.5)  {\small{$\alpha_3$}};

    \node (o) at (15, 7.25) {\small{$\xcancel{\alpha_{6}} $}};

      \node (p) at (15,9) [rrt] {\small{$\alpha_{12}$}};

            \node (q) at (15,2) [rrt] {\small{$\alpha_{12}$}};

                \node (r) at (13.5,-1.5) [td2] {\small{$\alpha_2$}};

    \node (s) at (16.5,-1.5) [brt2] {\small{$\alpha_3$}};

    \node (t) at (15, 0.25) [rrt] {\small{$\alpha_{6} $}};

  \draw (a) -- (e) -- (h);
    \draw (b) -- (e)  ;

  \draw (f) -- (h) ;

  \draw (b) -- (f);
  \draw (c) -- (f) -- (j);

  \draw (c) -- (g) -- (j);

  \draw (d) -- (g)  ;

   \draw (h) -- (l)  ;

    \draw (j) -- (l)  ;

        \draw[->] (7.5, 3.75) -- (10.5, 6.50)  ;

            \draw[->] (7.5, 3.75) -- (10.5, 1)  ;

            \draw (r) -- (t)  ;

    \draw (s) -- (t)  ;

\node at (d) {};

\end{tikzpicture}
\caption{A pictorial representation of the calculation of the
characters in $\Irr(U)_{\{\alpha_{12}\}}$ for $G$ of type $\rF_4$.}
\label{fig:a12}

\end{center}
\end{figure}

In Figure \ref{fig:a12}, we illustrate how we have calculated these characters.
The roots in a circle are in $\cZ$; the roots in a straight box are in $\cL$ and the roots in a
dotted box are in $\cA$.

\subsection{The algorithm.}

Our algorithm is used to calculate $\Irr(U)_\Sigma$ for each antichain
$\Sigma$ in $\Phi^+$.  We explain the algorithm below, which is written in a sort of pseudocode;
the comments in {\em italics} aim to make it easier to understand.

\begin{algorithm} \label{alg:main}
$ $

\medskip

\noindent
{\bf INPUT:}
\begin{itemize}
\item $\Phi^+ = \{\alpha_1,\dots,\alpha_N\}$, the set of positive roots of
a root system with a fixed enumeration such that $i \le j$ whenever $\alpha_i \le \alpha_j$.
\item $\Sigma$, an antichain in $\Phi^+$.
\end{itemize}

\medskip

\noindent
{\bf VARIABLES:}
\begin{itemize}
\item $\cS \sub \Phi^+$ is a quattern.
\item $\cZ$ is a subset of $\cZ(\cS)$.
\item $\cA \sub \Phi^+$ keeps a record of the roots $\beta$ used in a TYPE R reduction.
\item $\cL \sub \Phi^+$ keeps a record of the roots $\delta$ used in a TYPE R reduction.
\item $\cK \sub \Phi^+$ keeps a record of the roots indexing root subgroups in the quotient of the associated quattern group.
\item $\fS$ is a stack of tuples of the form $(\cS,\cZ,\cA,\cL,\cK)$ as above to be considered later in the algorithm.
\item $\fO = (\fO_1,\fO_2)$ is the output.
\end{itemize}

\medskip

\noindent
{\bf INITIALIZATION:}
\begin{itemize}
\item $\cK := \cK_\Sigma$.
\item $\cS := \Phi^+ \setminus \cK_{\Sigma}$.
\item $\cZ := \Sigma$.
\item $\cA := \varnothing$.
\item $\cL := \varnothing$.
\item $\fS := \varnothing$.
\item $\fO := (\varnothing,\varnothing)$.
\end{itemize}

\medskip

\noindent
{\em During the algorithm we consider $\Irr(X_\cS)_\cZ$, going
into four subroutines called ``ABELIAN CORE'', ``TYPE R'', ``TYPE S''
and ``NONABELIAN CORE''.  }

\medskip

\noindent
{\bf ABELIAN CORE.}
\begin{algorithmic}
\IF{$\cS = \cZ(\cS)$}
\STATE Adjoin $(\cS,\cZ,\cA,\cL,\cK)$ to $\fO_1$.
\STATE {\em In this case $X_\cS$ is abelian and we can parameterize the characters in $\Irr(X_\cS)_\cZ$.}
\IF{$\fS = \varnothing$}
\STATE {\bf Finish} and {\bf output} $\fO$.
\STATE {\em In this case we have no more characters to consider, so we are done.}
\ELSE
\STATE Remove the tuple at the
top of the stack $\fS$ and replace $(\cS,\cZ,\cA,\cL,\cK)$ with it, and go to ABELIAN CORE.
\ENDIF
\ELSE
\STATE Go to TYPE R.
\ENDIF
\end{algorithmic}

\medskip

\noindent
{\bf TYPE R.}
\begin{algorithmic}
\STATE Look for pairs $(\beta,\delta) =(\alpha_i,\alpha_j)$ that satisfy the
conditions of Lemma \ref{lem:small} for some $\gamma \in \cZ$.
\IF{such a pair $(\alpha_i,\alpha_j)$ exists}
\STATE Choose the pair with $j$ maximal, and update the variables as follows.
\begin{itemize}
\item $\cS := \cS \setminus \{\alpha_i,\alpha_j\}$.
\item $\cA := \cA \cup \{\alpha_i\}$.
\item $\cL := \cL \cup \{\alpha_j\}$.
\item $\cK := \cK \cup \{\alpha_j\}$.
\end{itemize}
\STATE {\em We are replacing $\cS$ with $\cS'$ as in Lemma \ref{lem:small}, and recording this in $\cA$, $\cL$ and $\cK$.}
\STATE Go to ABELIAN CORE.
\ELSE
\STATE Go to  TYPE S.
\ENDIF
\end{algorithmic}

\medskip

\noindent
{\bf TYPE S.}
\begin{algorithmic}
\IF{ $\cZ(\cS) \setminus (\cZ \cup \cD(\cS)) \ne \varnothing$}
\STATE Let $i$ be maximal such that $\alpha_i \in \cZ(\cS) \setminus (\cZ \cup \cD(\cS))$, and update
as follows.
\begin{itemize}
\item $\fS := \fS \cup \{(\cS \setminus \{\alpha_i\},\cZ, \cA, \cL, \cK \cup \{\alpha_i\})\}$.
\item $\cZ := \cZ \cup \{\alpha_i\}$.
\end{itemize}
\STATE {\em Here we are using Lemma \ref{lem:split}. We first add $(\cS \setminus \{\alpha_i\},\cZ, \cA, \cL, \cK \cup \{\alpha_i\})$ to the stack to be considered later, recording that $X_{\alpha_i}$ is in the kernel of these characters by adding $\alpha_i$ to $\cK$. Then we replace $(\cS,\cZ, \cA, \cL, \cK)$ with $(\cS,\cZ \cup \{\alpha_i\}, \cA, \cL, \cK)$
for the current run.}
\STATE Go to ABELIAN CORE.
\ELSE
\STATE Go to NONABELIAN CORE
\ENDIF
\end{algorithmic}

\medskip

\noindent
{\bf NONABELIAN CORE.}
\begin{algorithmic}
\STATE Adjoin $(\cS,\cZ,\cA,\cL,\cK)$ to $\fO_2$.
\STATE {\em We are no longer able to apply reductions of TYPE R or of TYPE S, and
$X_\cS$ is not abelian, so the algorithm gives up, and this case is output
as a nonabelian core as discussed further later.}
\IF{$\fS = \varnothing$}
\STATE {\bf Output} $\fO$ and {\bf finish}.
\STATE {\em In this case we have no more characters to consider, so we are done.}
\ELSE
\STATE Remove the tuple at the
top of the stack $\fS$ and replace $(\cS,\cZ,\cA,\cL,\cK)$ with it, and go to ABELIAN CORE.
\ENDIF
\end{algorithmic}

\end{algorithm}

The letter R in the TYPE R reduction means ``reduction lemma", while the letter S in TYPE S means ``split".
The letters $\cA$ and $\cL$ mean ``arm" and ``leg" respectively; this terminology is used in \cite{HLMsr},
and it is motivated by the fact that each pair $(\beta,\delta)$ gives rise to a so-called ``hook" subgroup.

We move on to discuss how we interpret the output.  We begin by defining
what we mean by a \emph{core}, which is an element of the output
of our algorithm.

\begin{definition}
Let us suppose that Algorithm \ref{alg:main} has run with input $(\Phi^+,\Sigma)$ and given output $\fO$.
\begin{itemize}
\item An element $(\cS,\cZ,\cA,\cL,\cK)$ of $\fO_1$ is called
an {\em abelian core} for $\Irr(U)_\Sigma$.
\item An element $(\cS,\cZ,\cA,\cL,\cK)$ of $\fO_2$ is called
a {\em nonabelian core} for $\Irr(U)_\Sigma$.
\end{itemize}
\end{definition}

We discuss how we can determine the characters in $\Irr(U)_\Sigma$
corresponding to a core $\fC = (\cS,\cZ,\cA,\cL,\cK)$ in $\fO_1 \cup \fO_2$.
In particular, when $\fC \in \fO_1$ is an abelian core, we give a complete
description of the irreducible characters, however for nonabelian cores there is more work
required.  We require some notation for what occurs in the algorithm.

We obtain $\fC$ through a sequence of reductions of TYPE R and
of TYPE S applied in Algorithm \ref{alg:main}; though here we only
need to consider the TYPE S reduction in this sequence if a root $\gamma$ is added
to $\cK$ (rather than to $\cZ$).  So we consider the sequence of reductions
where in each one either:
\begin{itemize}
\item a pair of roots $\beta$ and $\delta$ is taken from
$\cS$ in a TYPE R reduction, and $\beta$ is added to $\cA$ and $\delta$ is added
to $\cL$ and $\cK$; or
\item a root $\gamma$ is taken from $\cS$ and added to $\cK$.
\end{itemize}
We let $\ell = \ell_\fC$ be the number of these reductions, and define the sequence
$
T(\fC) = (t_1,\dots,t_\ell),
$
where $t_i = \rR$ if the $i$th reduction is a TYPE R reduction
and $t_i = \rS$ if the $i$th reduction is a TYPE S reduction.
We let $I(\rR,\fC)$ be the set of $i$ such that $t_i = \rR$
and $I(\rS,\fC)$ be the set of $i$ such that $t_i = \rS$.
For $i \in I(\rR,\fC)$ we write $(\beta_i,\delta_i)$ for the
pair of roots used in the TYPE R reduction, and for
$i \in I(\rS,\fC)$, we write $\gamma_i$ for the root added
to $\cK$ in the TYPE S reduction.  Thus we have $\cA = \{\beta_i \mid i \in I(\rR,\fC)\}$,
$\cL = \{\delta_i \mid i \in I(\rR,\fC)\}$
and $\cK \setminus \cK_\Sigma = \cL \cup \{\gamma_i \mid i \in I(\rS,\fC)\}$.

We also define the subsets $\cP^0,\cP^1,\dots,\cP^\ell$
and $\cK^0,\cK^1,\dots,\cK^\ell$ of $\Phi^+$ recursively by
\begin{itemize}
\item[] $\cP^0 = \Phi^+$ and $\cK^0 = \cK_\Sigma$;
\item[] $\cP^i =
\begin{cases}
\cP^{i-1} \setminus \{\beta_i\} & \text{if $t_i = \rR$} \\
\cP^{i-1} & \text{if $t_i = \rS$}
\end{cases}
$
\item[] $\cK^i =
\begin{cases}
\cK^{i-1} \cup \{\delta_i\} & \text{if $t_i = \rR$} \\
\cK^{i-1} \cup \{\gamma_i\} & \text{if $t_i = \rS$}
\end{cases}
$
\end{itemize}
We have the following lemma about these sets.

\begin{lemma} \label{lem:descent}
For each $i, j = 1,\dots,\ell$ with $i \le j$, we have that  $\cP^j$ is a closed set, and $\cK^i$ is normal in $\cP^j$.
In particular, $\cS^{i, j}=\cP^{j} \setminus \cK^i$ are quatterns.
\end{lemma}

\begin{proof}
Of course, $\cP^0=\Phi^+$ is closed. Let us assume that $\cP^{i-1}$ is closed. Without loss of generality, let $\cP^i=\cP^{i-1} \setminus \{\beta_i\}$.
For $\alpha, \alpha' \in \cP^i$, it cannot be that $\alpha+\alpha'=\beta_i$ by construction of $\cP^i$. Also, by inductive assumption,
we have that $\alpha + \alpha' \in \cP^{i-1}$ if $\alpha+\alpha'$ is a positive root. This implies $\alpha+\alpha' \in \cP^i$ or
$\alpha+\alpha' \notin \Phi^+$, that is, $\cP^i$ is closed.

To prove that $\cK^i$ is normal in $\cP^j$ for $i \le j$, it is enough to prove that $\cK^i$ is normal in $\cP^i$, since $\cK^i \subseteq \cP^j \subseteq \cP^i$.
Let $\alpha \in \cP^i$ and $\eta \in \cK^i$. Recall that $\eta \in \cK_{\Sigma}$ or $\eta$ is of the form $\gamma_k$ or $\delta_k$ as above for some $k \le i$.
If $\eta \in \cK_\Sigma$, then since $\cK_{\Sigma} \trianglelefteq \Phi^+$ we have that $\alpha+\eta \in \cK_{\Sigma}$ whenever $\alpha+\eta \in \Phi^+$.
If $\eta= \gamma_k$ for some $k \le i$, then $\eta$ is a central root in $\cS^{k-1, k-1} \supseteq \cS^{i, i}$, therefore since $\alpha \in \cP^i$ we have that
$\alpha + \eta \in \cK^{k-1} \subseteq \cK^i$ or $\alpha + \eta \notin \Phi^+$. If $\eta=\delta_k$, then we notice that  $\beta_k \notin \cP^k$, thus $\beta_k \notin \cP^i$, therefore if $\alpha \in \cP^i$ then
$\alpha + \eta \in \cK^{k-1}$ or $\alpha + \eta \notin \Phi^+$. This implies that $\cK^i$ is normal in $\cP^i$.
\end{proof}

Let $\psi \in \Irr(X_\cS)$.
We define characters $\overline \psi_i \in \Irr(X_{\cP^i\setminus \cK^i})$ for
$i=\ell,\ell-1,\dots,1,0$ recursively by the following sequence of inflations and inductions.
\begin{itemize}
\item[] $\overline \psi_{\ell} = \psi$
\item[] $\overline \psi_{i-1} =
\begin{cases}
\Ind^{\beta_i} \Inf_{\delta_i} \overline \psi_i & \text{if $t_i = \rR$} \\
\Inf_{\gamma_i} \overline \psi_i & \text{if $t_i = \rS$}
\end{cases}
$
\end{itemize}
Finally, we let $\overline \psi = \Inf_{\cK_\Sigma} \overline \psi_0 \in \Irr(U)$.

Suppose that $\fC = (\cS,\cZ,\cA,\cL,\cK) \in \fO_1$ is an abelian core.
We let $\cZ = \{\alpha_{i_1},\dots,\alpha_{i_m}\}$ and $\cS \setminus \cZ =
\{\alpha_{j_1},\dots,\alpha_{j_n}\}$.  Then we have
$$
\Irr(X_\cS)_\cZ = \{\lambda_{\ub}^{\ua}
\mid \ua = (a_{i_1}, \dots, a_{i_m}) \in (\F_q^\times)^m, \, \ub = (b_{j_1}, \dots, b_{j_n}) \in \F_q^n\},
$$
where $\lambda_\ub^\ua$ is defined by
$$
\lambda_\ub^\ua(x_{\alpha_{i_k}}(t))=\phi(a_{i_k}t) \quad  \text{ and } \quad  \lambda_\ub^\ua(x_{\alpha_{j_h}}(t))=\phi(b_{j_h}t)
$$
for every $k=1, \dots, m$ and $h=1, \dots, n$. We define $\chi_\ub^\ua = \overline {\lambda_\ub^\ua}$
and
$$
\Irr(U)_\fC = \{ \chi_\ub^\ua \mid \ua = (a_{i_1}, \dots, a_{i_m}) \in (\F_q^\times)^m, \ub = (b_{j_1}, \dots, b_{j_n}) \in \F_q^n\}.
$$
Through the bijections given by Lemmas \ref{lem:small} and \ref{lem:split}, this is precisely
the set of characters in $\Irr(U)_\Sigma$ corresponding to $\mathfrak C$.

We move on to consider a nonabelian core $\fC = (\cS,\cZ,\cA,\cL,\cK) \in \mathfrak O_2$.
In this case $X_\cS$ is not abelian, so we do not immediately have a parametrization of
$\Irr(X_\cS)_\cZ$, and it is necessary for us to determine a parametrization by hand.
We suppose this has been done and we have
$$
\Irr(X_\cS)_\cZ = \{\psi_\uc \mid \uc \in J_\fC\},
$$
where $J_\fC$ is some indexing set.
We define $\chi_\uc = \overline {\psi_\uc}$
and
$$
\Irr(U)_\fC = \{ \chi_\uc \mid \uc \in J_\fC \}.
$$
The aim of the next section is to develop
a method towards a determination of the set $J_\fC$
when $\fC$ is a nonabelian core.

From the comments given within Algorithm \ref{alg:main} and the discussion above, we deduce the
following theorem regarding the validity of our algorithm.

\begin{theorem}
Suppose that Algorithm \ref{alg:main} has run
with input $(\Phi^+,\Sigma)$ and given output $\fO = (\fO_1,\fO_2)$.
Then we have
$$
\Irr(U)_\Sigma = \bigsqcup_{\fC \in \fO_1} \Irr(U)_\fC \sqcup \bigsqcup_{\fC \in \fO_2} \Irr(U)_\fC.
$$
\end{theorem}

We note that the definitions of $\chi_\ub^\ua$
and $\chi_\uc$
given above involve a potentially very long sequence of inflations and inductions.
In fact it turns out that we can obtain them by a single inflation followed by a
single induction, which is stated in Theorem \ref{thm:1go} below.

To prove this theorem, we require the following lemma.  In the statement
of the lemma, we use the notation $\cA_i = \{\beta_j \mid j \ge i\}$,
$\cL_i = \{\delta_j \mid j \ge i\}$
and $\cK_i = \{ \gamma_j \mid j \ge i\}$.

\begin{lemma} \label{lem:1go}
Let $\psi \in \Irr(X_\cS)$, and for $i = 0,1,\dots,\ell$ define $\psi_i$ as above.
Then we have
$$
\psi_i = \Ind^{\cA_i} \Inf_{\cL_i \cup \cK_i} \psi.
$$
\end{lemma}

\begin{proof}
We prove this by reverse induction on $i$, the case $i = \ell$ being trivial.

The inductive step boils down to showing that
$$
\Inf_{\delta_i} \Ind^{\cA_{i+1}} =  \Ind^{\cA_{i+1}} \Inf_{\delta_i}
$$
if $t_i = \rR$ and showing that
$$
\Inf_{\gamma_i} \Ind^{\cA_{i+1}} =  \Ind^{\cA_{i+1}} \Inf_{\gamma_i}
$$
if $t_i = \rS$.
Thanks to Lemma \ref{lem:descent}, we are able to apply \eqref{eq:swap} to deduce both of these
equalities.
\end{proof}

\begin{theorem} \label{thm:1go}  $ $
\begin{itemize}
\item[(a)] Let $\fC \in (\cS,\cZ,\cA,\cL,\cK) \in \fO_1$ be an abelian core,
and let $\chi_\ub^\ua \in \Irr(U)_\fC$ be defined as above.
Then
$$
\chi_\ub^\ua = \Ind^\cA \Inf_{\cK} \lambda_\ub^\ua.
$$
In particular, $\chi_\ub^\ua$ is induced from a linear character of $X_{\cS \cup \cK}$.
\item[(b)]  Let $\fC \in (\cS,\cZ,\cA,\cL,\cK) \in \fO_2$ be a nonabelian core,
and let $\chi_\uc \in \Irr(U)_\fC$ be defined as above.
Then
$$
\chi_\uc = \Ind^\cA \Inf_{\cK} \psi_\uc.
$$
\end{itemize}
\end{theorem}

\begin{proof}
We only prove (a) as the proof of (b) is entirely similar.

By Lemma \ref{lem:1go}, we have that $(\chi_\ub^\ua)_0 = \Ind^\cA \Inf_{(\cK \setminus \cK_\Sigma)} \lambda_\ub^\ua$.
Thus
$$
\chi_\ub^\ua = \Inf_{\cK_\Sigma} \Ind^\cA \Inf_{(\cK \setminus \cK_\Sigma)} \lambda_\ub^\ua.
$$
Now we can apply
\eqref{eq:swap} to see that $\Inf_{\cK_\Sigma} \Ind^\cA  = \Ind^\cA \Inf_{\cK_\Sigma} $ from which we can deduce
the theorem.
\end{proof}

\begin{remark}
The choice of total order on $\Phi^+=\{\alpha_1, \dots, \alpha_N\}$ has a significant effect on how the algorithm
runs, as this is used to determine which reductions to make when there
may be a choice.  The resulting parametrization of $\Irr(U)_\Sigma$
consequently depends on this choice of enumeration.
\end{remark}

\begin{remark}
We make a slight abuse in the notation $\chi_{\ub}^{\ua}$. In fact, each $a_i$ and $b_j$ is supposed to record not just a
value in $\mathbb{F}_q^{\times}$ and $\mathbb{F}_q$ respectively, but also $i$ and $j$, so that $\chi_{\ub}^{\ua}$
should strictly read $\chi_{((j_1, b_{j_1}), \dots, (j_n, b_{j_n}))}^{((i_1, a_{i_1}), \dots, (i_m, a_{i_m}))}$, for corresponding choices of
$i_1, \dots, i_m$ and $j_1, \dots, j_n$ indexing positive roots.
\end{remark}

\subsection{Results of algorithm}

We have implemented Algorithm \ref{alg:main} in the algebra system GAP3 \cite{GAP3}, using the CHEVIE package \cite{CHEVIE}.
The algorithm requires us to just work
with $\Phi^+$ and the GAP commands for root systems allow us to do this. We use the enumeration of $\Phi^+$ as given in GAP.

We have run the GAP program for $G$ of rank less than or equal to 7.  For ranks less than or equal to 4 we are able
to deduce a complete parametrization, as the number of nonabelian cores is low.  More
specifically, for $G$ of rank $3$ or less, or $G$ of type $\rC_4$, there are no nonabelian cores, whilst for
 the types $\rB_4$ and $\rD_4$ there is one nonabelian core each, and in type $\rF_4$ we find six nonabelian cores.
The nonabelian core for type $\rD_4$ has already been encountered in
\cite{HLMd4}, and the core for type $\rB_4$ has the same representation theory for $p \ne 2$
as the one in type $D_4$, so these cores have been analysed.
The nonabelian cores for type $\rF_4$ are analysed in \S\ref{ss:f4cors}, and
the corresponding irreducible characters are determined.
The resulting parameterizations of irreducible characters of
$U_{\rB_4}$, $U_{\rC_4}$ and $U_{\rF_4}$ are tabulated in the appendix.
The parametrization for $U_{\rD_4}$ is contained in
\cite{HLMd4}.  Also we note that the parametrization of irreducible characters for $U_{\rB_3}$  can
be read off from that for $U_{\rB_4}$, as $U_{\rB_3}$ is a quotient of $U_{\rB_4}$.  Similarly,
the parametrization of irreducible characters of
$U_{\rC_3}$ can be read off from that for $U_{\rC_4}$.

From the parameterizations we can determine the number of irreducible
characters of $U$ of a given fixed degree.
In particular, we observe that if $G$ is of rank at most $4$ and $p$ is good,
then all irreducible characters of $U$ are of degree $q^d$ for some $d \in \Z_{\ge 0}$.
Moreover, the numbers of irreducible characters of $U$ of degree $q^d$,
is given by a polynomial in $q$ and these polynomials are the same
as the ones given in \cite[Table 3]{GMR2}, where they were only known to be valid for
$p \ge h$ (the Coxeter number of $G$).
Further we also obtain expressions as polynomials in $q$
for the number of characters of a given degree for type $\rF_4$,
and $p =3$; these are given in Table \ref{tab:nof4}. The case of type
$\rD_4$ and $p=2$ is covered in \cite{HLMd4}.

For $G$ or rank greater than 4, the number of nonabelian cores grows, so it is necessary
to develop an approach to deal with these in a systematic way. This should be based on
our analysis of nonabelian cores in the next section and is a
topic for future research. In Table \ref{tab:resultscores}, we present
the output from the algorithm including the number of nonabelian cores.

\begin{table}[t]
\renewcommand{\arraystretch}{1.4}
\begin{tabular}{|l|l|}
\hline
$d$ & $k(U_{\rF_4}(q), d)$  \\
\hline
$1$ & $v^4 +4v^3 +6v^2 +4v+1$   \\
\hline
$q$ & $v^5 +6v^4 +13v^3 +12v^2 +4v$   \\
\hline
$q^2$ & $v^6 +7v^5 +20v^4 +28v^3 +18v^2 +4v$   \\
\hline
$q^3$ & $4v^5 +20v^4 +33v^3 +21v^2 +4v$   \\
\hline
$q^4/3$ & $0$, \text{ if }$p \ge 5$ \\
 & $9v^4/2$, \text{ if }$p=3$ \\
\hline
$q^4$ & $v^8 +8v^7 +28v^6 +58v^5 +79v^4 +66v^3 +24v^2 +2v$, \text{ if }$p \ge 5$ \\
& $v^8 +8v^7 +28v^6 +59v^5 +161v^4/2 +67v^3 +24v^2 +2v$, \text{ if }$p=3$ \\
\hline
$q^5$ & $v^7 +7v^6 +22v^5 +39v^4 +37v^3 +15v^2+2v$, \text{ if }$p \ge 5$ \\
&  $v^7 +7v^6 +23v^5 +41v^4 +37v^3 +15v^2+2v$, \text{ if }$p=3$ \\
\hline
$q^6$ & $2v^6 +14v^5 +36v^4 +40v^3 +17v^2 +2v$, \text{ if }$p \ge 5$  \\
&  $2v^6 +14v^5 +36v^4 +39v^3 +17v^2 +2v$, \text{ if }$p=3$ \\
\hline
$q^7$ & $2v^6 +13v^5 +32v^4 +34v^3 +13v^2 +2v $   \\
\hline
$q^8$ & $4v^5+15v^4+19v^3+8v^2$   \\
\hline
$q^9$ & $v^5 +7v^4 +11v^3 +5v^2$  \\
\hline
$q^{10}$ & $v^4+3v^3+v^2$   \\
\hline
\hline
\multicolumn{2}{|c|}{$k(U_{\rF_4}(q))=\begin{cases} v^8 +9v^7 +40v^6 +124v^5 +256v^4 +288v^3 +140v^2 +24v+1, & \text{ if } p \ge 5 \\
v^8 +9v^7 +40v^6 +126v^5 +264v^4 +288v^3 +140v^2 +24v+1, & \text{ if } p = 3
\end{cases}
$} \\

\hline
\end{tabular}
\medskip
\caption{Numbers of irreducible characters of $U_{\rF_4}$ of fixed degree, for $v=q-1$ and $p \ne 2$.}
\label{tab:nof4}
\end{table}

\begin{table}[t]
\begin{tabular}{|l|l|l|l|l|}
\hline
Type & Antichains & Abelian cores & Nonabelian cores & Running time \\
\hline
\hline
$\rB_4$ & 70 & 80 & 1 (1.23\%) & $T \ll 1$ sec \\
\hline
$\rC_4$ & 70 & 90 & 0 (0\%) & $T \ll 1$ sec \\
\hline
$\rD_4$ & 50 & 52 & 1 (1.88\%) & $T \ll 1$ sec \\
\hline
$\rF_4$ & 105 & 177 & 6 (3.28\%) & $T\sim 1$ sec \\
\hline
$\rB_5$ & 252 & 358 & 10 (2.72\%) & $T\sim 3$ sec \\
\hline
$\rC_5$ & 252 & 417 & 1 (0.24\%) & $T\sim 3$ sec \\
\hline
$\rD_5$ & 182 & 214 & 7 (3.17\%) & $T\sim 1$ sec \\
\hline
$\rB_6$ & 924 & 1842 & 95 (4.90\%) & $T\sim 30$ sec \\
\hline
$\rC_6$ & 924 & 2254 & 22 (0.97\%) & $T\sim 30$ sec \\
\hline
$\rD_6$ & 672 & 991 & 55 (5.26\%) & $T\sim 10$ sec \\
\hline
$\rE_6$ & 833 & 1656 & 156 (8.61\%) & $T\sim 30$ sec \\
\hline
$\rB_7$ & 3432 & 11240 & 969 (7.94\%) & $T\sim 7$ min \\
\hline
$\rC_7$ & 3432 & 14216 & 294 (2.03\%) & $T\sim 7$ min \\
\hline
$\rD_7$ & 2508 & 5479 & 531 (8.84\%) & $T\sim 2.5$ min \\
\hline
$\rE_7$ & 4160 & 33594 & 7798 (18.84\%) & $T\sim 10$ min \\
\hline

\end{tabular}
\medskip
\caption{Results of the algorithm applied in types $\rB_i, \rC_i$ and $\rD_i$, $i=4, 5, 6, 7$ and $\rF_4$, $\rE_k$, $k=6, 7$.}
\label{tab:resultscores}
\end{table}

\section{Nonabelian cores}\label{sec:cors}

In this section we explain the methods we employ to analyse nonabelian cores.  It is helpful
for us first to deal with certain 3-dimensional groups that arise in our analysis.  Then
we outline our general method to deal with nonabelian cores, before explaining how this is applied to
the nonabelian cores in types $\rB_4$ and $\rF_4$.

\subsection{Some 3-dimensional groups} \label{sub:3dim}

Let $f : \F_q \times \F_q \to \F_q$ be an $\F_p$-bilinear map, which we assume
to be surjective.  We define the
group $V = V_f$ to be generated by subgroups
$X_1 = \{x_1(t) \mid t \in \F_q\} \cong \F_q$, $X_2 = \{x_2(t) \mid t \in \F_q\} \cong \F_q$ and
$Z = \{z(t) \mid t \in \F_q\} \cong \F_q$, subject to $Z \sub Z(V)$ and $[x_1(s),x_2(t)] = z(f(s,t))$.
In particular, throughout this subsection, $X_1$ will not denote the root subgroup $X_{\alpha_1}$, similarly for $X_2$.
It is straightforward to see that $V$ is a nilpotent group and that $V = X_1X_2Z$.
Moreover, our assumption that $f$ is surjective implies that
the derived subgroup of $V$ is $Z$.  We note that $V$ is not necessarily a special
group as its center can be strictly larger than its derived subgroup.

We quickly explain how to construct the irreducible characters of $V$.

First we note that the linear characters are given by the characters of $V/Z \cong X_1 \times X_2$.
For $b_1,b_2 \in \F_q$, we define $\chi_{b_1,b_2} \in \Irr(V)$ by
$\chi_{b_1,b_2}(x_1(s_1)x_2(s_2)z(t)) = \phi(b_1s_1+b_2s_2)$, so that
the linear characters of $V$ are $\{\chi_{b_1,b_2} \mid b_1,b_2 \in \F_q\}$.

For each $a \in \F_q^\times$, we define the linear character
$\lambda^a \in \Irr(Z)$ by $\lambda^a(z(t)) = \phi(at)$.
We analyse the characters in $\Irr(V \mid \lambda^a)$ using Lemma \ref{lem:rl}.
We define $X_1' = \{x_1(s) \in X_1 \mid \Tr(af(s,t)) = 0 \text{ for all } t \in \F_q\}$
and define $X_2'$ similarly.
Note that $X_1'$ and $X_2'$ may depend on $a$.
Since the map $\F_q \times \F_q \to \F_p$ given by $(s,t) \mapsto \Tr(f(s,t))$
is $\F_p$-bilinear, we deduce that $X_1'$ and $X_2'$ are
$\F_p$-subspaces of $X_1 \cong \F_q$ and $X_2 \cong \F_q$ respectively, and
we have $|X_1'| = |X_2'|$.  Thus we can choose complements $\tilde X_1$ and $\tilde X_2$
of $X_1'$ and $X_2'$ in $X_1$ and $X_2$ respectively.
Now by Lemma \ref{lem:rl} (with $X = \tilde X_1$ and $Y = \tilde Y_1$),
we see that $\psi \mapsto \Ind_{X_1'X_2Z}^{V} \Inf_{X_1'X_2Z/(\tilde{X}_2 \ker \lambda^a)}^{X_1'X_2Z} \psi$
gives a bijection from $\Irr(X_1'X_2Z/(\tilde{X}_2 \ker \lambda^a) \mid \lambda^a)$ to $\Irr(V \mid \lambda^a)$.
Finally, we observe that $X_1'X_2Z/(\tilde{X}_2 \ker \lambda^a)$ is abelian, so
$\Irr(X_1'X_2Z/(\tilde{X}_2 \ker \lambda^a) \mid \lambda^a)$ is in bijection with the set $\Irr(X_1' \times X_2')$, which
is easily described.

Before considering some particular choices of $f$, we note that it is
straightforward to show that $V_f \cong V_{af}$ for $a \in \F_q^\times$ by
reparamaterizing $Z$ appropriately.

For $f(s,t) = st$, we see that $V_f$ is isomorphic
to $U_{A_2}$.  Clearly we get $X' = Y' = 1$
and then
$$
\Irr(V) =  \{\chi_{b_1,b_2} \mid b_1,b_2 \in \F_q\} \cup \{\chi^a \mid a \in \F_q^\times\},
$$
where $\chi^a = \Ind_{X_2Z}^V \Inf_Z^{X_2Z} \lambda^a$.  Similarly for $f(s,t) = s^pt$
or $f(s,t) = (s^p-ds)t$ where $d \in \F_q^\times$  is not a $(p-1)$th power,
we see that $X' = Y' = 1$, and $\Irr(V)$ is given as above.

The case of major interest to us here is $f(s,t) = (s^p-ds)t$ where $d \in \F_q^\times$ is a $(p-1)$-th power, say $d = e^{p-1}$.
Then we find that $X_1' = \{x_1(s) \mid s^p-ds=0\} = \{x_1(s) \mid s \in e\F_p\}$ and $X_2' = \{x_2(t) \mid \Tr(at\T_e) = 0\} =
\{x_2(t) \mid t \in (e^{-p}/a) \F_p\}$ using Lemma \ref{lem:TL}.  Now for $c_1,c_2 \in \F_p$ we define the
characters $\lambda^a_{c_1,c_2} \in \Irr(X_1'X_2Z/(\tilde X_2 \ker \lambda^a))$ by
$$
\lambda^a_{c_1,c_2}(x_1(es_1)x_2((e^{-p}/a)s_2)z(t)) =
\phi(c_1s_1+c_2s_2+at).
$$
for every $s_1, s_2 \in \F_p$ and $t \in \F_q$. Then we have
$$
\Irr(V) =  \{\chi_{b_1,b_2} \mid b_1,b_2 \in \F_q\} \cup \{\chi^a_{c_1,c_2} \mid a \in \F_q^\times, c_1,c_2 \in \F_p\},
$$
where $\chi^a_{c_1,c_2} = \Ind_{X_1'X_2Z}^{V} \Inf_{X_1'X_2Z/(\tilde X_2 \ker \lambda^a)}^{X_1'X_2Z} \lambda^a_{c_1,c_2}$.
In this case, we get $q^2$ linear characters and $p^2(q-1)$ characters of degree $q/p$.

\subsection{A method for analysing nonabelian cores.} \label{sub:analyse}

Let $\fC = (\mathcal{S}, \mathcal{Z}, \cA, \cL, \cK)$ be a nonabelian core.
The set $\cS$ is a quattern corresponding to the pattern $\Phi^+ \setminus \cA$ and its normal subset $ \cK$. Further, we have
$\cZ=\cZ(\cS) \setminus \cD(\cS)$ as $\fC$ is a nonabelian core,
and we let $\cZ = \{\alpha_{i_1},\dots,\alpha_{i_m}\}$.
For each $\ua=(a_{i_1}, \dots, a_{i_m})\in (\mathbb{F}_q^{\times})^m$, we define $\mu=\mu^{\ua}: X_\cZ \rightarrow \F_q$ by
$\mu(x_{i_h}(t))=a_{i_h}t$ for $h=1,\dots,m$.  Then $\lambda=\lambda^{\ua} = \phi \circ \mu^\ua$ is a linear character of
$X_{\cZ}$.

We give a method to analyse the characters in $\Irr(X_\cS \mid \lambda)$.  We note that
the nature of the resulting parametrization and construction of the characters
may depend on the choice of $\ua$, and we see instances of this dependence in \S\ref{ss:f4cors}.
Further we remark that we do not assert that this method is guaranteed to work for every
nonabelian core, though it does apply for all the cores that we consider in \S\ref{ss:f4cors}.

We set $V = X_\cS/\ker \mu$ and $Z = X_\cZ/\ker \mu$.  Since $\ker \mu \sub \ker \lambda$, we have that
$\lambda$ factors through $Z$ and we also write $\lambda$ for this character of $Z$.
Then we have a bijection between $\Irr(V \mid \lambda)$ and $\Irr(X_\cS \mid \lambda)$ by inflating over $\ker \mu$,
and we work in $\Irr(V \mid \lambda)$ rather than in $\Irr(X_\cS \mid \lambda)$.  Given $\alpha \in \cS \setminus \cZ$
we identify $X_\alpha$ with its image in $V$.

We aim to find subsets $\cI$ and $\cJ$ of $\cS \setminus \cZ$ such that the following hold.
\begin{itemize}
\item $|\cI| = |\cJ|$;
\item $H = X_{\cS \setminus (\cI \cup \cZ)}Z$ is a subgroup of $V$;
\item $Y = X_\cJ \le Z(H)$; and
\item $Y Z$ is a normal subgroup of $V$.
\end{itemize}
We note that this implies that
\begin{itemize}
\item $X = X_\cI$ is a transversal of $H$ in $V$.
\end{itemize}
We would like to apply Lemma \ref{lem:rl} (the reduction lemma), and conditions (i)--(iv) do hold,
but condition (v) may not be satisfied, so we aim to adapt the situation slightly.

We consider the inflation  $\hat{\mu}$ of $\mu$ to $YZ$ and let $\hat \lambda = \phi \circ \hat \mu$ be the
inflation of $\lambda$ to $YZ$.
For $v \in V$, we consider the map $\psi_v : Y \to \F_q$ given by $\psi_v(y)=\hat{\mu}([v, y])$.
Since $Y \le Z(H)$ and $YZ \unlhd V$, we deduce from the commutator
relations that $\psi_v$ is $\F_q$-linear.
We let
$$
Y' = \bigcap_{v \in V} \ker(\psi_v) = \{y \in Y \mid {}^v\hat{\mu}(y)=\hat{\mu}(y) \text{ for all } v \in V\}.
$$
Then $Y'$ is an $\F_q$-subspace of $Y \cong \F_q^{|\cJ|}$.
Also, we define
$$
\tilde H =\mathrm{Stab}_V(\hat \mu) = \{v \in V \mid {}^v\hat{\mu}=\hat \mu\}.
$$
Then
$\tilde{H}$ is a subgroup of $V$ and $\tilde{H}=X'H$ for
$X'=\{x \in X \mid {}^x\hat{\mu}=\hat{\mu}\}$.

To prove that $X'$ and $Y'$ have the same cardinality we assume, for the rest of this subsection, that
$$
W=\{\psi_v \mid v \in V\} \text{ is an } \F_q\text{-subspace of the dual space } \mathrm{Hom}(Y, \F_q).
$$
This condition is easily checked to hold for all nonabelian cores that we examine when $G$ is of rank $4$
by looking at the form of \eqref{eq:commu} defined below in each of these cases.

\begin{lemma} \label{lem:orbit}
$|X'|=|Y'|$.
\end{lemma}

\begin{proof}
We have that the annihilator $\mathrm{Ann}_Y(W)$ of $W$ is $Y'$ by definition. Hence we have $\dim Y = \dim Y' + \dim W$, that is, $|Y|/|Y'|=|W|$.

We denote the $V$-orbit of $\hat{\mu}$ in $\mathrm{Hom}(YZ, \F_q)$ by $\hat{\mu}^V$. For $v, v' \in V$ we have that
$\psi_v=\psi_{v'}$ if and only if $\hat{\mu}([v, y])=\hat{\mu}([v', y])$ for all $y \in Y$ if and only if $\hat{\mu}(y^v)=\hat{\mu}(y^{v'})$ for all $y \in Y$.
Then the map
\begin{align*}
\hat{\mu}^V& \longrightarrow W  \\
\hat{\mu}^v & \mapsto \psi_v
\end{align*}
is well-defined and injective. It is clear that it is also surjective. Therefore, we have $|W|=|\hat{\mu}^V|$. Now by the
orbit-stabilizer theorem  we have that
$$|\hat{\mu}^V|=|V|/\mathrm{Stab}_V(\hat{\mu})=|V|/|\tilde{H}|=|X|/|X'|.$$
Combining the above equalities, we get
$$|Y|/|Y'|=|W|=|\hat{\mu}^V|=|X|/|X'|.$$
Since $|Y|=|X|$, the claim follows.
\end{proof}

Moreover, we have the following property about $X'$.

\begin{lemma} \label{lem:X'}
Let $x \in X$ be such that ${}^x\hat{\lambda}=\hat{\lambda}$. Then $x \in X'$.
\end{lemma}

\begin{proof} We show that for such $x$ we have ${}^x\hat{\mu}=\hat{\mu}$. The hypothesis is equivalent to
\begin{equation} \label{eq:last}
\phi \circ {}^x\hat{\mu} = \phi \circ \hat{\mu}, \text{ that is, } \phi \circ ({}^x\hat{\mu} - \hat{\mu})=1.
\end{equation}
For $y \in Y$ and $z \in Z$, we have
$${}^x\hat{\mu}(yz)=\hat{\mu}(y^xz^x)={}^x\hat{\mu}(y)+\hat{\mu}(z)=\hat{\mu}([y, x])+\mu(z)=-\psi_x(y)+\mu(z),$$
then using the assumptions that $Y \le Z(H)$ and $YZ \unlhd V$, we deduce from the commutator
relations that ${}^x\hat{\mu}$ is $\F_q$-linear. Hence ${}^x\hat{\mu} - \hat{\mu}$ is also $\F_q$-linear.
Therefore the image of ${}^x\hat{\mu} - \hat{\mu}$ is either $0$
or $\F_q$. But if it were $\F_q$, then \eqref{eq:last} would imply $\phi(c)=1$ for every $c \in \F_q$, which is a contradiction.
Thus we have ${}^x\hat{\mu}=\hat{\mu}$, so $x \in X'$.
\end{proof}

We write $\cI =\{\alpha_{i_1},\cdots,\alpha_{i_m}\}$ and
$\cJ = \{\alpha_{j_1},\cdots,\alpha_{j_m}\}$,
such that $i_1 \le \dots \le i_m$ and $j_1 \le \dots \le j_m$.
In general, $Y'$ and $X'$ can be determined by the following equation,
\begin{equation}\label{eq:commu}
\hat{\mu}([x_{\alpha_{j_1}}(s_{j_1}) \cdots x_{\alpha_{j_m}}(s_{j_m}), x_{\alpha_{i_1}}(t_{i_1}) \cdots x_{\alpha_{i_m}}(t_{i_m})])=0.
\end{equation}
We note that as the map $\psi_x$ for $x \in X$ is $\F_q$-linear, the left hand side of \eqref{eq:commu}
is linear in $s_{j_1}, \dots, s_{j_m}$.  Therefore,
the values of $s_{j_1}, \dots, s_{j_m}$ such that \eqref{eq:commu} holds
for every $t_{i_1}, \dots, t_{i_m}$ form an $\F_q$-subspace of $Y$, which determines $Y'$.

Under an additional assumption on $Y$, we are able to apply Lemma \ref{lem:rl} in
the following lemma.  We define $\bar H$ to be the preimage of $\tilde H$ in $X_\cS$.

\begin{lemma}\label{pr:reduction}
Suppose that there exists a subgroup $\tilde{Y}$ of $Y$ such that $Y=Y' \times \tilde{Y}$ and $[X, \tilde{Y}] \subseteq \tilde{Y}Z.$
Then we have a bijection
\begin{align*}
\Irr(\tilde H/\tilde Y \mid \lambda) & \to
\Irr(V \mid \lambda) \\
\chi &\mapsto \Ind_{\tilde H}^V \Inf_{\tilde H/\tilde Y}^{\tilde H} \chi
\end{align*}
Consequently we have a bijection
\begin{align*}
\Irr(\tilde H/\tilde Y \mid \lambda) & \to
\Irr(X_{\cS} \mid \lambda) \\
\chi &\mapsto \Ind_{\bar H}^{X_\cS} \Inf_{\tilde H/\tilde Y}^{\bar H} \chi
\end{align*}
\end{lemma}

\begin{proof}
We want to check that $\tilde{H}, \tilde{X}, \tilde{Y}$ and $Z$ satisfy all the assumptions of
Lemma \ref{lem:rl} as subgroups of $V$ with respect to $\lambda \in \Irr(Z)$.
Clearly we have that $Z \le Z(V)$ and $\tilde Y \cap Z=1$.
By assumption, we have that $X$ normalizes $\tilde{Y}Z$, and we have that $H$ centralizes $\tilde YZ$,
so $\tilde{Y}Z \trianglelefteq V$.
Since $\tilde{Y} \le Y \le Z(H)$, we have that $\tilde{Y}$ is normalized by $H$.
Moreover, if $x' \in X'$ and $y \in Y$, by definition of $X'$ we have that
$$\hat{\mu}(y^{-1}y^{x'})=\hat{\mu}(y^{-1})+\hat{\mu}(y^{x'})=0,$$
and since $\ker\hat{\mu}=Y \ker \mu$ we have that $X'$ normalizes $Y$. Along with the assumption that $[X,\tilde Y] \sub \tilde YZ$,
we deduce that $X'$ normalizes $\tilde Y$.  Hence $\tilde Y \trianglelefteq \tilde H$.

Now we are left to check condition (v) of the reduction lemma. We write
$\tilde \lambda \in \Irr(\tilde Y Z)$ for the
inflation of $\lambda$ to $\tilde Y Z$, and note that
$\tilde \lambda =\hat{\lambda}|_{\tilde Y Z}$. Let $\tilde{X}$ be
a transversal of $\tilde{H}$ in $V$. Assume that
${}^{\tilde{x}_1}\tilde \lambda={}^{\tilde{x}_2}\tilde{\lambda}$ for
$\tilde{x}_1, \tilde{x}_2 \in \tilde{X}$.
Let $y \in Y$ and $z \in Z$ and write $y=y'\tilde{y}$, where $y' \in Y'$ and $\tilde{y} \in \tilde{Y}$.
We have
\begin{align*}
{}^{\tilde{x}_1}\hat{\lambda}(y'\tilde{y}z)
&= \hat{\lambda}(y'^{\tilde{x}_1})\tilde{\lambda}(\tilde{y}^{\tilde{x}_1})\lambda(z) \\
&= \hat{\lambda}(y')({}^{\tilde{x}_1} \tilde{\lambda})(\tilde{y})\lambda(z) \\
&= \hat{\lambda}(y'^{\tilde{x}_2})({}^{\tilde{x}_2} \tilde{\lambda})(\tilde{y})\lambda(z) \\
&= \hat{\lambda}(y'^{\tilde{x}_2})\hat{\lambda}(\tilde{y}^{\tilde{x}_2})\lambda(z) \\
&= {}^{\tilde{x}_2}\hat{\lambda}(y'\tilde{y}z).
\end{align*}
In the above sequence of equalities we use that $\hat{\lambda}(y'^{\tilde{x}_1}) = \hat{\lambda}(y') = \hat{\lambda}(y'^{\tilde{x}_2})$
by definition of $Y'$, that $\tilde{y}^{\tilde{x}_1}, \tilde{y}^{\tilde{x}_2} \in \tilde{Y}Z$ since $[X, \tilde{Y}] \subseteq \tilde{Y}Z$,
and that ${}^{\tilde{x}_1}\tilde \lambda={}^{\tilde{x}_2}\tilde{\lambda}$ by assumption.
Hence we have ${}^{\tilde{x}_1\tilde{x}_2^{-1}}\hat \lambda = \hat \lambda$. By Lemma \ref{lem:X'}, this implies $\tilde{x}_1\tilde{x}_2^{-1} \in X'$ and
thus $\tilde{x}_1 = \tilde{x}_2$ as $\tilde{X}$ is a transversal of $\tilde{H}$ in $V$. By Lemma \ref{lem:orbit},
we have that $|X'|=|Y'|$. Thus we can apply Lemma \ref{lem:rl} to deduce the first bijection.

We can now apply Lemma \ref{lem:rl} also to deduce the second bijection.
\end{proof}

We note that if $[X, Y] \subseteq Z$, then we may take an arbitrary complement $\tilde Y$ of $Y'$ in $Y$,
and the assumption $[X, \tilde{Y}] \subseteq \tilde{Y}Z$
is obviously satisfied.

Also we note that the parametrization of characters resulting from Lemma \ref{pr:reduction} does not
actually depend on the choice of $\tilde Y$.  This can be shown by observing that
the restriction of $\Ind_{\tilde H}^V \Inf_{\tilde H/\tilde Y}^{\tilde H} \chi$ to
$Y'Z$ is a multiple of $\chi$ viewed as a character of $Y'Z$.

If $Y'$ is central in $\tilde H / \tilde Y$, then we can extend $\lambda \in \Irr(Z)$ to $Y'$. This
turns out to be useful when applying the reduction lemma again in
$\tilde H / \tilde Y$ in the analysis in \S\ref{ss:f4cors}.

\begin{remark} \label{rem:1go}
Suppose that Lemma \ref{pr:reduction} applies and let $\psi \in \Irr(\tilde H/\tilde Y \mid \lambda)$.
Then we have  that
$\Ind_{\bar H}^{X_\cS} \Inf_{\tilde H/\tilde Y}^{\bar H} \psi \in \Irr(X_{\cS})$,
and
$$
\overline \psi = \Ind_{X_{\cS \cup \cK}}^U \Inf_{X_\cS}^{X_{\cS \cup \cK}} \Ind_{\bar H}^{X_\cS} \Inf_{\tilde H/\tilde Y}^{\bar H} \psi \in \Irr(U)_\fC
$$
by Theorem \ref{thm:1go}.
Since $X_{\cK} \trianglelefteq U$, we have that
$\bar{H}X_{\cK}$ is a subgroup of $X_{\cS \cup \cK}$,
and we have $X_{\cK} \trianglelefteq \bar HX_{\cK} $.
Then using similar arguments
to those in the proof of Lemma \ref{lem:1go}, we can show that
$$
\overline{\psi}= \Ind_{\bar{H}X_{\cK}}^U \Inf_{\tilde{H}/\tilde{Y}}^{\bar{H}X_{\cK}} \psi.
$$
In \S\ref{ss:f4cors}, we apply this argument (sometimes iteratively) to show that each irreducible character
considered there can be obtained as an induced character of a linear character.

Consider the case that $Y' = 1$ for all choices of $\lambda$, and that $Y$ is
normal in $\bar H$. We have $\bar H/Y = X_{\cS \setminus (\cI \cup \cJ)}$.
Defining
$$
\overline{\psi}= \Ind^{\cA \cup \cI} \Inf_{\cK \cup \cJ} \psi
$$
for $\psi \in \Irr(X_{\cS \setminus (\cI \cup \cJ)})_{\cZ}$ sets up a bijection
from $\Irr(X_{\cS \setminus (\cI \cup \cJ)})_{\cZ}$ to $\Irr(U)_\fC$.
\end{remark}

\subsection{The nonabelian cores for types $\rB_4$ and $\rF_4$}  \label{ss:f4cors}

For $G$ of type $\rB_4$ there is one nonabelian core and for $G$ of type $\rF_4$, there are six nonabelian cores.
We analyse these case by case using the method given in \S\ref{sub:analyse}, and we use
the notation introduced there.

For the nonabelian core in $U_{\rB_4}$ and for one of the nonabelian cores in $U_{\rF_4}$,
we find $\cI$ and $\cJ$ such that $Y' = 1$ for all $\ua \in (\F_q^\times)^m$, and $Y$ is normal in $\bar H$ with $\bar H/Y$ abelian.
For such a core $\fC$ we let $\cS \setminus (\cI \cup \cJ) = \{\alpha_{h_1},\dots,\alpha_{h_n}\}$, and
for $\ub = (b_{h_1},\dots,b_{h_n}) \in \F_q^n$, we let $\lambda_\ub^\ua \in \Irr(H)$ be the linear character
defined by $\lambda_\ub^\ua(x_{\alpha_{h_j}}(t)) = \phi(b_{h_j}t)$ for $j=1, \dots, n$, and
$\lambda_\ub^\ua|_{X_\cZ} = \lambda^\ua$.
Let
$$
\chi_\ub^\ua = \Ind^{\cA \cup \cI} \Inf_{\cK \cup \cJ} \lambda_\ub^\ua.
$$
Then using Remark \ref{rem:1go}, we see that
$$
\Irr(U)_\fC = \{\chi_\ub^\ua \mid \ua \in (\F_q^\times)^m, \ub \in \F_q^n\}.
$$
For these cores we include no further details, and just give $\cI$ and $\cJ$
in the tables in the appendix.

Below we consider the remaining nonabelian cores in $U = U_{\rF_4}$.
We denote these cores by $\fC^1$, $\fC^2$, $\fC^3$, $\fC^4$ and $\fC^5$.
For each $\fC^i = (\cS,\cZ,\cA,\cL,\cK)$ we give
$\cS$, $\cZ$, $\cA$ and $\cL$; we note that $\cK$ can then
easily be determined. Then we analyse $\Irr(X_\cS)_\cZ$ before explaining
how this parameterizes $\Irr(U)_{\fC^i}$ and how these characters
can be obtained by inducing linear characters using Lemma \ref{pr:reduction}
and Remark \ref{rem:1go}.

We notice that if $\mathfrak{C}=(\cS, \cZ, \cA, \cL, \cK)$ and $\mathfrak{C'}=(\cS', \cZ', \cA', \cL', \cK')$
are cores of $U_{\rm{F_4}}$, then $(|\cS|, |\cZ|) \ne (|\cS'|, |\cZ'|)$. In particular, $X_{\cS}$ is not isomorphic to $X_{\cS'}$.

\medskip
\noindent {\bf The nonabelian core in $\fC^1$.}
This core occurs for  $\Sigma=\{\alpha_{11}, \alpha_{13}\}$, and we
have
\begin{itemize}
\item $\mathcal{S}=\{\alpha_1,\alpha_{2},\alpha_4,\alpha_{5},\alpha_{6},\alpha_{7},\alpha_{9},\alpha_{10},\alpha_{11},\alpha_{13}\}$,
\item $\mathcal{Z}=\{\alpha_{5},\alpha_{10},\alpha_{11},\alpha_{13}\}$,
\item $\cA = \{\alpha_3\}$ and
\item $\cL = \{\alpha_8\}$.
\end{itemize}
Using the method of \S\ref{sub:analyse}, we take
\begin{itemize}
\item $Y=X_2X_6X_9$,
\item $X=X_1X_4X_7$ and then we have that
\item $H = YZ$.
\end{itemize}
In this case (\ref{eq:commu}) is
$$
s_2(-a_5t_1+a_{10}t_7)+s_6(a_{10}t_4-a_{13}t_7)+s_9(-a_{11}t_1+a_{13}t_4) = 0.
$$

For $a_{11} \ne a_5 a_{13}^2/a_{10}^2$, we have $Y'=1$ and $Y$ is
normal in $\bar H$.
Then as explained in Remark \ref{rem:1go} we get the family of characters
$$
\Irr(U)_{\fC^1}^1=\{\chi^{a_5, a_{10}, a_{11}^*, a_{13}} \mid a_5, a_{10}, a_{11}^*, a_{13} \in \mathbb{F}_q^{\times}, a_{11}^* \ne a_5
(a_{13}/a_{10})^2\} \sub \Irr(U)_{\fC^1},
$$
where
$$
\chi^{a_5, a_{10}, a_{11}^*, a_{13}} = \Ind^{\cA \cup \cI} \Inf_{\cK \cup \cJ} \lambda^{a_5, a_{10}, a_{11}^*, a_{13}}.
$$
We have that $\Irr(U)_{\fC^1}^1$ consists of $(q-1)^3(q-2)$ characters of degree $q^4$.

For $a_{11} = a_5 a_{13}^2/a_{10}^2$, we have $X'=X_{1,4,7}=\{x_{1,4,7}(t) \mid t \in \F_q\}$ and $Y'=X_{2,6,9} =\{x_{2,6,9}(s) \mid s \in \F_q\}$, where
$$
x_{1,4,7}(t) =x_1(a_{10}^2t)x_4(a_5a_{13}t)x_7(a_5a_{10}t) \,\, \text{ and } \,\, x_{2,6,9}(s)=x_2(a_{13}^2s)x_6(a_{10}a_{13}s)x_9(-a_{10}^2s).
$$
We can take any complement of $Y'$ in $Y$, and we choose $\tilde Y= X_2X_9$.
Then we have $\tilde{H}/\tilde{Y}=X'Y'Z$, which is abelian. We denote by $\lambda^{a_5, a_{10}, a_{13}}$ the character $\lambda^{a_5, a_{10}, a_{11}, a_{13}}$
with $a_{11} = a_5 a_{13}^2/a_{10}^2$. For $b_{1,4,7}, b_{2,6,9} \in \F_q$, we define $\lambda_{b_{1,4,7}, b_{2,6,9}}^{a_5, a_{10}, a_{13}} \in \Irr(X'Y'Z)$ by
extending $\lambda^{a_5, a_{10}, a_{13}}$, and setting $\lambda_{b_{1,4,7}, b_{2,6,9}}^{a_5, a_{10}, a_{13}}(x_{1,4,7}(t)) = \phi(b_{1,4,7}t)$ and
$\lambda_{b_{1,4,7}, b_{2,6,9}}^{a_5, a_{10}, a_{13}}(x_{2,6,9}(t)) = \phi(b_{2,6,9}t)$ for every $t \in \F_q$.
Then as explained in Remark \ref{rem:1go} we get the family of characters
$$
\Irr(U)_{\fC^1}^2=\{\chi_{b_{1,4,7}, b_{2,6,9}}^{a_5,a_{10},a_{13}} \mid a_5,a_{10},a_{13} \in
\mathbb{F}_q^{\times}, b_{1,4,7}, b_{2,6,9} \in \mathbb{F}_q\},
$$
where
$$
\chi_{b_{1,4,7},b_{2,6,9}}^{a_5,a_{10},a_{13}} = \Ind_{\bar H X_\cK}^U \Inf_{\tilde{H}/\tilde{Y}}^{\bar H X_\cK}
\lambda_{b_{1,4,7}, b_{2,6,9}}^{a_5, a_{10}, a_{13}}.
$$
We have that $\Irr(U)_{\fC^1}^2$ consists
of $q^2(q-1)^3$ characters of degree $q^3$.

We have that  $\Irr(U)_{\fC^1} = \Irr(U)_{\fC^1}^1 \cup \Irr(U)_{\fC^1}^2$ and
this gives all the irreducible characters corresponding to $\fC^1$.

\medskip
\noindent {\bf The nonabelian core $\fC^2$.}
This core occurs for  $\Sigma=\{\alpha_{12}, \alpha_{16}\}$, and we
have
\begin{itemize}
\item $\cS=\{\alpha_1,\alpha_{2},\alpha_3,\alpha_{5},\alpha_{6},\alpha_{7},\alpha_8,\alpha_{9},\alpha_{10},\alpha_{12},\alpha_{16}\}$,
\item $\cZ=\{\alpha_8,\alpha_9,\alpha_{12},\alpha_{16}\}$,
\item $\cA = \{\alpha_4\}$ and
\item $\cL = \{\alpha_{13}\}$.
\end{itemize}
Using the method of \S\ref{sub:analyse}, we take
\begin{itemize}
\item $Y=X_5 X_6 X_{10}$,
\item $X=X_1X_3X_7$ and then we have that
\item $H=X_2YZ$.
\end{itemize}
In this case (\ref{eq:commu}) is
$$
s_5(a_8t_3+a_{12}t_7)+s_6(-a_8t_1-2a_9t_3)+s_{10}(-a_{12}t_1+2a_{16}t_7)=0.
$$

For $a_{16} \ne a_9 a_{12}^2/a_8^2$, we have $Y'=1$ and $Y$ is normal in $\bar H$. Further, $\bar H/Y = X_2X_\cZ$, so
as explained in Remark \ref{rem:1go} we get the family of characters
$$
\Irr(U)_{\fC^2}^1=\{\chi_{b_2}^{a_8, a_{9}, a_{12}, a_{16}^*} \mid a_8, a_{9}, a_{12}, a_{16}^* \in \mathbb{F}_q^{\times}, a_{16}^* \ne a_9
(a_{12}/a_8)^2, b_2 \in \mathbb{F}_q\} \sub \Irr(U)_{\fC^2},
$$
where
$$
\chi_{b_2}^{a_8, a_{9}, a_{12}, a_{16}^*}= \Ind^{\cA \cup \cI} \Inf_{\cK \cup \cJ} \lambda_{b_2}^{a_8, a_{9}, a_{12}, a_{16}^*},
$$
and $\lambda_{b_2}^{a_8, a_{9}, a_{12}, a_{16}^*} \in \Irr(\bar H/Y)$ is defined in the usual way.
We have that $\Irr(U)_{\fC^2}^1$ consists of $q(q-1)^3(q-2)$ characters of degree $q^4$.

Now suppose $a_{16} = a_9 a_{12}^2/a_8^2$. We have $X' = \{x_{1,3,7}(t) \mid t \in \F_q\}$ and
$Y' = \{x_{5,6,10}(s) \mid s \in \F_q\}$, where
$$
x_{1,3,7}(t)=x_1(2a_9a_{12}t)x_3(-a_8a_{12}t)x_7(a_8^2t)\,\, \text{ and } \,\, x_{5, 6, 10}(s)=x_5(2a_9a_{12}s)x_6(a_8a_{12}s)x_{10}(-a_8^2s).
$$
We can take any complement of $Y'$ in $Y$ and we choose $\tilde Y= X_5X_{10}$.
Then we have $\tilde{H}/\tilde{Y}=X_2X'YZ/\tilde{Y}$ and $Y' \subseteq Z(\tilde{H}/\tilde{Y})$. From now on, we denote by $\lambda^{a_8, a_{9}, a_{12}}$
the character $\lambda^{a_8, a_{9}, a_{12}, a_{16}}$ with $a_{16} = a_9 a_{12}^2/a_8^2$.

A computation in $\tilde{H}/\tilde{Y}$ gives
$$
[x_2(s), x_{1,3,7}(t)]=x_{5, 6, 10}(-st).
$$
Therefore, $\tilde{H}/\tilde{Y}$ is the direct product of $Z$ and $X_2X'Y/\tilde Y$.
Further $X_2X'Y/\tilde Y$ is isomorphic to the three-dimensional group $V_f$ for $f(s, t)=-st$ from \S\ref{sub:3dim}.

We label the linear characters of $X_2X'Y/\tilde Y$ by $\chi_{b_2,b_{1,3,7}}$.
By tensoring these characters with $\lambda^{a_8, a_{9},a_{12}}$ and then
applying $\Ind_{\bar H X_\cK}^U \Inf_{\tilde{H}/\tilde{Y}}^{\bar H X_\cK}$ we obtain the
family of characters
$$
\Irr(U)_{\fC^2}^2=\{\chi_{b_2, b_{1, 3, 7}}^{a_8, a_{9},  a_{12}} \mid a_8, a_{9}, a_{12} \in
\mathbb{F}_q^{\times}, b_2, b_{1, 3, 7} \in \F_q\},
$$
which consists of $q^2(q-1)^3$ characters of degree $q^3$.

Let $a_{5, 6, 10} \in \F_q^{\times}$. We write $\lambda^{a_8, a_{9}, a_{12},a_{5,6,10}}$ for the linear character of $Y'Z$
defined by extending $\lambda^{a_8, a_9, a_{12}}$ to $Y'$ in the usual way.  By applying $\Ind_{X'YX_\cZ X_\cK}^U \Inf_{YZ/\tilde{Y}}^{X'YX_\cZ X_\cK}$
to these linear characters we obtain the family of characters
$$
\Irr(U)_{\fC^2}^3=\{\chi^{a_8, a_{9}, a_{12}, a_{5, 6, 10}} \mid a_8, a_{9}, a_{12}, a_{5, 6, 10} \in \F_q^{\times}\},
$$
which consists of $(q-1)^4$ characters of degree $q^4$.

We have $\Irr(U)_{\fC^2} = \Irr(U)_{\fC^2}^1 \cup \Irr(U)_{\fC^2}^2 \cup \Irr(U)_{\fC^2}^3$ and this gives all the
irreducible characters corresponding to $\fC^2$.

\medskip
\noindent {\bf The nonabelian core $\fC^3$.}
This core occurs for  $\Sigma=\{\alpha_{14}, \alpha_{15}\}$, and we
have
\begin{itemize}
\item $\cS=\{\alpha_2,\alpha_4,\alpha_6,\alpha_7,\alpha_8,\alpha_{10},\alpha_{11},\alpha_{14},\alpha_{15}\}$,
\item $\cZ=\{\alpha_{10},\alpha_{14},\alpha_{15}\}$,
\item $\cA = \{\alpha_1,\alpha_3,\alpha_5\}$ and
\item $\cL = \{\alpha_9,\alpha_{12},\alpha_{13}\}$.
\end{itemize}
Using the method of \S\ref{sub:analyse}, we take
\begin{itemize}
\item $Y=X_6 X_7 X_{11}$,
\item $X=X_2X_4X_8$ and then we have that
\item $H=YZ$.
\end{itemize}
In this case (\ref{eq:commu}) is
$$
s_6(a_{10}t_4+2a_{14}t_8)+s_7(-a_{10}t_2+a_{15}t_8)+s_{11}(-a_{14}t_2+a_{15}t_4)=0.
$$

For $p \ge 5$, we have $Y'=1$ and $Y$ is normal in $\bar H$.  So
as explained in Remark \ref{rem:1go} we obtain
$$
\Irr(U)_{\fC^3}^{p \ge 5}=\{\chi^{a_{10}, a_{14}, a_{15}} \mid  a_{10}, a_{14}, a_{15} \in \F_q^{\times} \}
$$
by applying $\Ind^{\cA \cup \cI} \Inf_{\cK \cup \cJ}$ to the characters in $\Irr(X_{\cS \setminus (\cI \cup \cJ)})_\cZ$.
We have that $\Irr(U)_{\fC^3}^{p \ge 5}$ consists of $(q-1)^3$ characters of degree $q^6$.

Now suppose $p=3$. We have $X'=\{x_{2, 4, 8}(t) \mid t \in \F_q\}$ and $Y'=\{x_{6, 7, 11}(s) \mid s \in \F_q\}$, where
$$
x_{2,4,8}(t)=x_2(a_{15}t)x_4(a_{14}t)x_8(a_{10}t) \,\, \text{ and } \,\, x_{6, 7, 11}(s)=x_6(a_{15}s)x_7(a_{14}s)x_{11}(-a_{10}s).
$$
We can take $\tilde Y = X_6X_{11}$, and we have that $\tilde H/\tilde Y \cong X'Y'Z$ is abelian. This yields
$$
\Irr(U)_{\fC^3}^{p =3} =\{\chi_{b_{2, 4, 8}, b_{6, 7, 11}}^{a_{10}, a_{14}, a_{15}} \mid  a_{10}, a_{14}, a_{15} \in \mathbb{F}_q^{\times}, b_{2,
4, 8}, b_{6, 7, 11} \in \mathbb{F}_q \},
$$
where these characters are obtained by applying $\Ind_{\bar H X_\cK}^U \Inf_{\tilde H/\tilde Y}^{\bar H X_\cK} $ to the linear characters
$\lambda_{b_{2, 4, 8}, b_{6, 7, 11}}^{a_{10}, a_{14}, a_{15}}$ of $\tilde H/\tilde Y$, which are labelled in the usual way.
We have that $\Irr(U)_{\fC^3}^{p = 3}$ consists of $q^2(q-1)^3$ characters of degree $q^5$.

\medskip
\noindent {\bf The nonabelian core $\fC^4$.}
This core occurs for  $\Sigma=\{\alpha_{11}, \alpha_{12}, \alpha_{13}\}$, and we
have
\begin{itemize}
\item $\cS=\{ \alpha_{1}, \dots, \alpha_{13}\}$,
\item $\cZ=\Sigma=\{\alpha_{11}, \alpha_{12}, \alpha_{13}\}$,
\item $\cA = \varnothing$ and
\item $\cL = \varnothing$.
\end{itemize}
Using the method of \S\ref{sub:analyse}, we take
\begin{itemize}
\item $Y=X_5X_8X_9X_{10}$,
\item $X=X_{1}X_{3}X_{4}X_{7}$ and then we have that
\item $H=X_2X_6YZ$.
\end{itemize}
In this case \eqref{eq:commu} is
$$
s_5(-a_{11}t_3^2+a_{12}t_7)+s_8(-2a_{11}t_3+a_{12}t_4)+s_9(-a_{11}t_1+a_{13}t_4)+s_{10}(-a_{12}t_1-a_{13}t_3) = 0.
$$

For $p \ge 5$, we have that $Y'=1$, and $Y$ is normal in $\bar H$.  Also we have $\bar H/Y \cong X_2X_6X_\cZ$ is abelian.  For
$b_2, b_6 \in \F_q$ we let $\lambda_{b_2, b_6}^{a_{11}, a_{12}, a_{13}} \in \Irr(X_{\cS \setminus (\cI \cup \cJ)})_\cZ$ be the linear
character with the usual notation.
Then as explained in Remark \ref{rem:1go} we obtain
$$
\Irr(U)_{\fC^4}^{p \ge 5}=\{\chi_{b_2, b_6}^{a_{11}, a_{12}, a_{13}} \mid  a_{11}, a_{12}, a_{13} \in \mathbb{F}_q^{\times},  b_2, b_6 \in
\F_q\},
$$
where $\chi_{b_2, b_6}^{a_{11}, a_{12}, a_{13}} = \Ind^{\cA \cup \cI} \Inf_{\cK \cup \cJ} \lambda_{b_2, b_6}^{a_{11}, a_{12}, a_{13}}$.
We have that $\Irr(U)_{\fC^4}^{p \ge 5}$ is
a family of $q^2(q-1)^3$ characters of degree $q^4$.

Now suppose $p=3$.  We have $X'=\{x_{1,3,4, 7}(t) \mid t \in \F_q\}$ and
$Y'=\{x_{8, 9, 10}(s) \mid s \in \mathbb{F}_q\}$, where
$$
x_{1,3,4,7}(t)=x_1(a_{13}t)x_3(-a_{12}t)x_4(a_{11}t)x_7(-a_{11}a_{12}t^2) \text{ and }
$$
$$
x_{8,9,10}(s)=x_8(a_{13}s)x_9(-a_{12}s)x_{10}(a_{11}s).
$$
We can take $\tilde Y = X_5X_8X_9$, and we have $\tilde H/\tilde Y =X_2X_6X'YZ/\tilde{Y}$. By Lemma \ref{pr:reduction},
we have $\Irr(V \mid \lambda)$ is in bijection with $\Irr(\tilde H/\tilde Y \mid \lambda)$.

We continue by considering $\Irr(\tilde H/\tilde Y \mid \lambda)$ and note that $Y'$ lies in
the centre of $\tilde{H}/\tilde{Y}$.  For $a_{8,9,10} \in \F_q^\times$, we let $\lambda^{a_{8,9,10}}$
be the extension of $\lambda$ to $Y'Z$ with $\lambda^{a_{8,9,10}}(x_{8,9,10}(t)) = \phi(a_{8,9,10}t)$ for every $t \in \F_q$.
Then $\Irr(\tilde H/\tilde Y \mid \lambda)$ decomposes as the union of $\Irr(\tilde H/\tilde Y \mid \lambda^{a_{8,9,10}})$
over $a_{8,9,10} \in \F_q^\times$ along with $\Irr(\tilde H/Y \mid \lambda)$.

A computation in
$\tilde{H}/\tilde{Y}$ gives
$$
[x_6(s), x_{1,3,4,7}(t)]=x_{8, 9, 10}(st).
$$
Now by Lemma \ref{lem:rl}, we have that $\Irr(\tilde H/\tilde Y \mid \lambda^{a_{8,9,10}})$
is in bijection with $\Irr(X_2YZ/\tilde Y \mid \lambda^{a_{8,9,10}})$.  Further, we have that
$X_2YZ/\tilde Y \cong X_2Y'Z$ is abelian, and we label the linear characters in
$\Irr(X_2YZ/\tilde Y \mid \lambda^{a_{8,9,10}})$ as $\lambda_{b_2}^{a_{11}, a_{12}, a_{13}, a_{8,9,10}}$ in the usual way.
This gives the family of characters
$$
\Irr(U)_{\fC_4}^{1, p=3}=\{\chi_{b_2}^{a_{11}, a_{12}, a_{13}, a_{8,9,10}} \mid a_{11}, a_{12}, a_{13}, a_{8,9,10} \in \F_q^{\times}, b_2
\in \F_q\},
$$
where by Remark \ref{rem:1go} we have $\chi_{b_2}^{a_{11}, a_{12}, a_{13}, a_{8,9,10}} = \Ind_{X_2X'YX_\cZ X_\cK}^U \Inf_{X_2YZ/\tilde Y}^{X_2X'YX_\cZ X_\cK} \lambda_{b_2}^{a_{11}, a_{12}, a_{13}, a_{8,9,10}}$.
We have that $\Irr(U)_{\fC^4}^{1, p=3}$ consists of $q(q-1)^4$ irreducible characters of degree $q^4$.

It remains to consider $\Irr(\tilde H/Y \mid \lambda)$.  We have $\tilde H/Y = X_2X'X_6YZ/Y$ and
$X_6$ is central in $\tilde H/Y$.
For $a_6 \in \F_q^\times$,  we let $\mu^{a_6} : X_6Z \to \F_q$ be the extension of $\mu:Z \to \F_q$ to $X_6$ defined as usual, and $\lambda^{a_6} \in \Irr(X_6Z)$
be such that $\lambda^{a_6}=\phi \circ \mu^{a_6}$.
Then $\Irr(\tilde H/Y \mid \lambda)$ decomposes as the union of $\Irr(\tilde H/Y \mid \lambda^{a_6})$
over $a_6 \in \F_q^\times$ along with $\Irr(\tilde H/X_6Y \mid \lambda)$.

A computation in $\tilde H/Y$ gives
$$
[x_2(t), x_{1, 3, 4,7}(s)]=x_6(-a_{12}st)x_{11}(a_{12}^2a_{13}s^3t).
$$

We note that the quotient $\tilde H/(Y\ker \mu^{a_6}) = X_2X'X_6YZ/(Y\ker \mu^{a_6})$ is
isomorphic to the three-dimensional group $V_f$ where
$f(s, t)=a_{12}t(a_{11}a_{12}a_{13}s^3-a_6s)$ is as given in \S\ref{sub:3dim}, and
we have that $\Irr(\tilde H/Y \mid \lambda^{a_6})$ is in bijection with $\Irr(\tilde H/(Y\ker \mu^{a_6}) \mid \lambda^{a_6})$.
Thus we can apply the analysis of $\Irr(V_f)$ in \S\ref{sub:3dim}.  We let $d = a_6/a_{11}a_{12}a_{13}$.

Suppose first that $d$ is a square in $\F_q$.
In this case we write $a_{1,6}$ for $a_6$, and we define $e \in \F_q$ such that $e^2=d$.
We let
$$W' = \{x_{1, 3, 4,7}(es) \mid s \in \F_3\} \quad \text{ and } \quad W_2 = \{x_2((e^{-3}/a_{11}a_{12}^2a_{13})t) \mid t \in \F_3\},$$
and
we define $\lambda_{c_{1, 3, 4, 7}, c_2}^{a_{11}, a_{12}, a_{13}, a_{1,6}}$ for $c_{1, 3, 4, 7}, c_2 \in \F_3$ of $W_1W_2X_6YZ/(Y \ker \lambda^{a_{11}, a_{12}, a_{13}, a_{1,6}})$
as in \S\ref{sub:3dim}.
Then we get the family of characters
$$
\Irr(U)_{\fC^4}^{2, 1, p=3} = \{ \chi_{c_{1, 3, 4, 7}, c_2}^{a_{11}, a_{12}, a_{13}, a_{1,6}} \mid  a_{11}, a_{12}, a_{13} \in \F_q^{\times},
a_{1,6} \in a_{11}a_{12}a_{13}S_q,
c_{1, 3, 4, 7}, c_2 \in \F_3\},
$$
where
$$
\chi_{c_{1, 3, 4, 7}, c_2}^{a_{11}, a_{12}, a_{13}, a_{1,6}} = \Ind_{X'W_2X_6YX_\cZ X_\cK}^U \Inf_{W_1W_2X_6YZ/(Y \ker \lambda^{a_{11}, a_{12}, a_{13}, a_{1,6}})}^{X'W_2X_6YX_\cZ X_\cK}
\lambda_{c_{1, 3, 4, 7}, c_2}^{a_{11}, a_{12}, a_{13}, a_{1,6}}
$$
and $S_q$ denotes the set of nonzero squares in $\F_q$.
We have that $\Irr(U)_{\fC^4}^{2, 1, p=3}$ consists of $9(q-1)^4/2$ characters of degree $q^4/3$.

Suppose now that $d$
is a nonsquare in $\F_q$.  In this case we write $a_{2,6}$ for $a_6$. We write $\lambda^{a_{11}, a_{12}, a_{13}, a_{2, 6}}$ for the
linear characters of $X_6YZ/(Y \ker \lambda^{a_{11}, a_{12}, a_{13}, a_{2,6}})$ in the usual notation.  Then we get the family of characters
$$
\Irr(U)_{\fC^4}^{2, 2, p=3} = \{ \chi^{a_{11}, a_{12}, a_{13}, a_{2, 6}} \mid a_{2, 6} \in \F_q^\times \setminus (a_{11}a_{12}a_{13}S_q), a_{11}, a_{12}, a_{13} \in \F_q^{\times}\},
$$
where
$$
\chi^{a_{11}, a_{12}, a_{13}, a_{2, 6}} = \Ind_{X'X_6YX_\cZ X_\cK}^U \Inf_{X_6YZ/(Y \ker \lambda^{a_{11}, a_{12}, a_{13}, a_{2,6}})}^{X'X_6YX_\cZ X_\cK} \lambda^{a_{11}, a_{12}, a_{13}, a_{2, 6}}.
$$
We have that $\Irr(U)_{\fC^4}^{2, 2, p=3}$ consists of $(q-1)^4/2$ characters of degree $q^4$.

Similarly, we can analyse $\Irr(\tilde H/X_6Y \mid \lambda)$ using the arguments for the three-dimensional group $V_f$ where
$f(s, t)=a_{11}a_{12}^2a_{13}s^3t$.
Therefore, we get the family of characters
$$
\Irr(U)_{\fC^4}^{3, p=3} = \{ \chi^{a_{11}, a_{12}, a_{13}} \mid a_{11}, a_{12}, a_{13} \in \mathbb{F}_q^{\times}\},
$$
where the characters are given by
$$
\chi^{a_{11}, a_{12}, a_{13}} = \Ind_{X'X_6YX_\cZ X_\cK}^U \Inf_{X_6YZ/X_6Y}^{X'X_6YX_\cZ X_\cK} \lambda^{a_{11}, a_{12}, a_{13}}.
$$
We have that $\Irr(U)_{\fC^4}^{3, p=3}$ consists of $(q-1)^3$ characters of degree $q^4$.

Putting this together we obtain
$$
\Irr(U)_{\fC^4}^{p=3} = \Irr(U)_{\fC^4}^{1, p=3} \cup \Irr(U)_{\fC^4}^{2,1, p=3} \cup \Irr(U)_{\fC^4}^{2,2, p=3} \cup \Irr(U)_{\fC^4}^{3, p=3}.
$$

\medskip
\noindent {\bf The nonabelian core $\fC^5$.}
This core occurs for  $\Sigma=\{\alpha_{12}, \alpha_{13}, \alpha_{14}\}$, and we
have
\begin{itemize}
\item $\cS=\{\alpha_1,\alpha_3,\alpha_4,\alpha_5,\alpha_6,\alpha_7,\alpha_8,\alpha_9,\alpha_{10},\alpha_{12},\alpha_{13},\alpha_{14}\}$,
\item $\cZ=\Sigma=\{\alpha_{12}, \alpha_{13}, \alpha_{14}\}$,
\item $\cA = \{\alpha_2\}$ and
\item $\cL = \{\alpha_{11}\}$.
\end{itemize}
Using the method of \S\ref{sub:analyse}, we take
\begin{itemize}
\item $Y=X_1X_7X_8X_9$,
\item $X=X_4X_5X_6X_{10}$ and then we have that
\item $H=X_3YZ$.
\end{itemize}
In this case \eqref{eq:commu} is
$$
s_1(-a_{14}t_4^2+a_{12}t_{10})+s_7(-a_{12}t_5+a_{13}t_6)+s_8(a_{12}t_4-2a_{14}t_6)+s_9(a_{13}t_4+a_{14}t_5)=0.
$$

For $p \ge 5$, we have that $Y'=1$, and $Y$ is normal in $\bar H$.  Also we have $\bar H/Y \cong X_3X_\cZ$ is abelian.  For
$b_3 \in \F_q$ we let $\lambda_{b_3}^{a_{12}, a_{13}, a_{14}} \in \Irr(X_{\cS \setminus (\cI \cup \cJ)})_\cZ$ be the linear
character with the usual notation.
Then as explained in Remark \ref{rem:1go} we obtain
$$
\Irr(U)_{\fC^5}^{p \ge 5}=\{\chi_{b_3}^{a_{12}, a_{13}, a_{14}} \mid  a_{12}, a_{13}, a_{14} \in \mathbb{F}_q^{\times},  b_3 \in \mathbb{F}_q\},
$$
where $\chi_{b_3}^{a_{12}, a_{13}, a_{14}} = \Ind^{\cA \cup \cI} \Inf_{\cK \cup \cJ} \lambda_{b_3}^{a_{12}, a_{13}, a_{14}}$.
We have that $\Irr(U)_{\fC^5}^{p \ge 5}$ is a family of $q(q-1)^3$ characters of degree $q^5$.

Now suppose $p=3$.  We have $X'=\{x_{4,5,6,10}(t) \mid t \in \F_q\}$ and $Y'=\{x_{7, 8, 9}(s) \mid s \in \F_q\}$, where
$$
x_{4,5,6,10}=x_4(-a_{14}t)x_5(a_{13}t)x_6(a_{12}t)x_{10}(a_{12}a_{14}t^2) \text{ and }  x_{7, 8, 9}=x_7(a_{14}s)x_8(-a_{13}s)x_9(a_{12}s).
$$
We can take $\tilde Y = X_1X_7X_8$, and we have $\tilde{H}/\tilde{Y}=X_3X'YZ/\tilde{Y}$. By Lemma \ref{pr:reduction},
we have that $\Irr(V \mid \lambda)$ is in bijection with $\Irr(\tilde H/\tilde Y \mid \lambda)$.

A computation in $\tilde{H}/\tilde{Y}$ gives
$$
[x_3(s), x_{4, 5, 6, 10}(t)]=x_{7, 8, 9}(-st).
$$
We notice that $\tilde{H}/\tilde{Y}$ is the direct product of $Z$ and the 3-dimensional group $X_3X'Y/\tilde{Y} \cong X_3X'Y'$. Then the analysis in $\S \ref{sub:3dim}$ applies with $f(s, t)=-st$.

We label the linear characters of $X_3X'Y/\tilde{Y}$ by $\chi_{b_3,b_{4,5,6,10}}$.
By tensoring these characters with $\lambda^{a_{12}, a_{13}, a_{14}}$ and then
applying $\Ind_{\bar H X_\cK}^U \Inf_{\tilde{H}/\tilde{Y}}^{\bar H X_\cK}$ we obtain the
family of characters
$$
\Irr(U)_{\fC^5}^{1, p=3}=\{ \chi_{b_3, b_{4,5, 6, 10}}^{a_{12}, a_{13}, a_{14}} \mid a_{12}, a_{13}, a_{14} \in \F_q^{\times}, b_3, b_{4,
5, 6, 10} \in \F_q \},
$$
which consists of $q^2(q-1)^3$ characters of degree $q^4$.

Let us fix $a_{7, 8, 9} \in \F_q^{\times}$. We write $\lambda^{a_{12}, a_{13}, a_{14},a_{7, 8, 9}}$ for the linear character of $Y'Z$
defined in the usual way.  By applying $\Ind_{X'YX_\cZ X_\cK}^U \Inf_{YZ/\tilde{Y}}^{X'YX_\cZ X_\cK}$
to these linear characters we obtain the family of characters
$$
\Irr(U)_{\fC^5}^{2, p=3}=\{ \chi^{a_{12}, a_{13}, a_{14} a_{7, 8, 9}} \mid a_{12}, a_{13}, a_{14}, a_{7, 8, 9} \in \F_q^{\times}\},
$$
which consists of $(q-1)^4$ characters of degree $q^5$.

We have $\Irr(U)_{\fC^5}^{p=3} = \Irr(U)_{\fC^2}^{1,p=3} \cup \Irr(U)_{\fC^2}^{2,p=3}$.

\appendix

\section*{Appendix: Tables of parametrization of $\Irr(U)$}

This appendix contains a parametrization of the irreducible characters of $U_{\rB_4}$, $U_{\rC_4}$ and $U_{\rF_4}$
when $p$ is not a very bad prime for $U$, that is $p \ne 2$.

The notation in the tables is as follows. The first column corresponds to the families of the form
$\mathcal{F}_\Sigma$, where $\mathcal{F}_\Sigma$ is the family of irreducible characters of
$U$ arising from an antichain $\Sigma$. The second column contains character labels for those
families as explained in Sections \ref{sec:nonc} and \ref{sec:cors}. For a fixed core $(\cS, \cZ, \cA, \cL, \cK)$, we define
$$
I_\cA =\{i \in \{1, \dots, |\Phi^+|\} \mid \alpha_i \in \cA\},
$$
and define $I_{\cL}$ similarly. In case of nonabelian cores,
$I_{\cI}$ and
$I_{\cJ}$ are also defined in the same fashion.
The third column contains $I_{\cA}$ and $I_{\cL}$. We note that $\cK$ can be determined from $\cA$, $\cL$
and the labels of the characters.  For the abelian cores, we recall that the irreducible
characters in the family are obtained by applying $\Ind^\cA \Inf_\cK$ to the linear characters in $\Irr(X_\cS)_\cZ$.
We use the {\bf bold} font to identify nonabelian cores. In these cases, we also use the second column to give any
relation between the indices and the third column to give some information
on the construction of these characters.   In the case where we have $Y' = 1$ and $Y$ is normal in $\bar H$,
we give $I_\cI$ and $I_\cJ$ in the third column, as in this case the irreducible characters
are given by applying $\Ind^{\cA \cup \cI} \Inf_{\cK \cup \cJ}$ to
linear characters in $\Irr(X_{\cS \setminus (\cI \cup \cJ)})_\cZ$.  In other cases, we
refer the reader to the relevant part of \S\ref{ss:f4cors}.
Finally, the fourth column records the number of irreducible characters in a family corresponding
to some character labels, and the fifth column records their degree.

\medskip

\begin{center}
\textbf{Parametrization of the irreducible characters of $U_{\rB_4}$}
\end{center}

 \begin{center} \tiny
\begin{longtable}{|l|l|l|l|l|}
 \hline
$\mathcal F$ & $\chi$ & $I$ & Number & Degree \\
\hline
\hline
\endhead
\hline
\endfoot
\hline \caption{The parametrization of the irreducible characters of $U_{\rB_4}(q)$, where $q=p^e$ and $p \ge 3$.} \label{tab:b4}
\endlastfoot
$\mathcal{F}_{\text{lin}}$ & $\chi_{b_1, b_2, b_3, b_4}$ &  &$q^4$ & $1$ \\
\hline
$\mathcal{F}_{\mathbf{5}}$ & $\chi^{a_{5}}$ & \begin{tiny} $I_{\cA}=\{1\}$, \end{tiny} \begin{tiny}$I_{\cL}=\{2\}$, \end{tiny}  & $  q-1$ & $q$ \\
\hline
$\mathcal{F}_{\mathbf{6}}$ & $\chi^{a_{6}}$ & \begin{tiny} $I_{\cA}=\{2\}$, \end{tiny} \begin{tiny}$I_{\cL}=\{3\}$, \end{tiny}  & $  q-1$ & $q$ \\
\hline
$\mathcal{F}_{\mathbf{7}}$ & $\chi^{a_{7}}$ & \begin{tiny} $I_{\cA}=\{3\}$, \end{tiny} \begin{tiny}$I_{\cL}=\{4\}$, \end{tiny}  & $  q-1$ & $q$ \\
\hline
$\mathcal{F}_{\mathbf{8}}$ & $\chi_{b_{2}}^{a_{8}}$ & \begin{tiny} $I_{\cA}=\{1\}$, \end{tiny} \begin{tiny}$I_{\cL}=\{5\}$, \end{tiny}  & $q^{}(q-1)^{}$ & $q$ \\
\hline
$\mathcal{F}_{\mathbf{9}}$ & $\chi_{b_{2}}^{a_{9}}$ & \begin{tiny} $I_{\cA}=\{1,3\}$, \end{tiny} \begin{tiny}$I_{\cL}=\{5,6\}$, \end{tiny}  & $q^{}(q-1)^{}$ & $q^2$ \\
\hline
$\mathcal{F}_{\mathbf{10}}$ & $\chi_{b_{3}}^{a_{10}}$ & \begin{tiny} $I_{\cA}=\{2,4\}$, \end{tiny} \begin{tiny}$I_{\cL}=\{6,7\}$, \end{tiny}  & $q^{}(q-1)^{}$ & $q^2$ \\
\hline
$\mathcal{F}_{\mathbf{11}}$ & $\chi_{b_{2}, b_{5}, b_{6}}^{a_{11}}$ & \begin{tiny} $I_{\cA}=\{1,3\}$, \end{tiny} \begin{tiny}$I_{\cL}=\{8,9\}$, \end{tiny}  & $q^{3}(q-1)^{}$ & $q^2$ \\
\hline
 \end{longtable}

 \end{center}

\begin{center}\textbf{Parametrization of the irreducible characters of $U_{\rC_4}$} \end{center}

 \begin{center} \tiny
  \begin{longtable}{|l|l|l|l|l|}
 \hline
$\mathcal F$ & $\chi$ & $I$ & Number & Degree \\
\hline
\hline
\endhead
\hline
\endfoot
\hline \caption{The parametrization of the irreducible characters of $U_{\rC_4}(q)$, where $q=p^e$ and $p \ge 3$.} \label{tab:c4}
\endlastfoot
$\mathcal{F}_{\text{lin}}$ & $\chi_{b_1, b_2, b_3, b_4}$ &  & $q^4$& $1$\\
\hline
$\mathcal{F}_{\mathbf{5}}$ & $\chi^{a_{5}}$ & \begin{tiny} $I_{\cA}=\{1\}$, \end{tiny} \begin{tiny}$I_{\cL}=\{2\}$, \end{tiny}  & $  q-1$ & $q$ \\
\hline
$\mathcal{F}_{\mathbf{6}}$ & $\chi^{a_{6}}$ & \begin{tiny} $I_{\cA}=\{2\}$, \end{tiny} \begin{tiny}$I_{\cL}=\{3\}$, \end{tiny}  & $  q-1$ & $q$ \\
\hline
$\mathcal{F}_{\mathbf{7}}$ & $\chi^{a_{7}}$ & \begin{tiny} $I_{\cA}=\{3\}$, \end{tiny} \begin{tiny}$I_{\cL}=\{4\}$, \end{tiny}  & $  q-1$ & $q$ \\
\hline
$\mathcal{F}_{\mathbf{8}}$ & $\chi_{b_{1}}^{a_{8}}$ & \begin{tiny} $I_{\cA}=\{2\}$, \end{tiny} \begin{tiny}$I_{\cL}=\{5\}$, \end{tiny}  & $q^{}(q-1)^{}$ & $q$ \\
\hline
$\mathcal{F}_{\mathbf{9}}$ & $\chi_{b_{2}}^{a_{9}}$ & \begin{tiny} $I_{\cA}=\{1,3\}$, \end{tiny} \begin{tiny}$I_{\cL}=\{5,6\}$, \end{tiny}  & $q^{}(q-1)^{}$ & $q^2$ \\
\hline
$\mathcal{F}_{\mathbf{10}}$ & $\chi_{b_{3}}^{a_{10}}$ & \begin{tiny} $I_{\cA}=\{2,4\}$, \end{tiny} \begin{tiny}$I_{\cL}=\{6,7\}$, \end{tiny}  & $q^{}(q-1)^{}$ & $q^2$ \\
\hline
$\mathcal{F}_{\mathbf{11}}$ & $\chi_{b_{1}}^{a_{11}}$ & \begin{tiny} $I_{\cA}=\{2,3,6\}$, \end{tiny} \begin{tiny}$I_{\cL}=\{5,8,9\}$, \end{tiny}  & $q^{}(q-1)^{}$ & $q^3$ \\
\hline
\end{longtable}

\end{center}

\begin{center} \textbf{Parametrization of the irreducible characters of $U_{\rF_4}$} \end{center}

\begin{center} \tiny
 \begin{longtable}{|l|l|l|l|l|}
 \hline
$\mathcal F$ & $\chi$ & $I$ & Number & Degree \\
\hline \hline
\endhead
\hline
\endfoot
\hline \caption{The parametrization of the irreducible characters of $U_{\rF_4}(q)$, where $q=p^e$ and $p \ge 3$.} \label{tab:f4}
\endlastfoot

$\mathcal{F}_{\text{lin}}$ & $\chi_{b_1, b_2, b_3, b_4}$ &  & $q^4$ & $1$ \\
\hline
$\mathcal{F}_{\mathbf{5}}$ & $\chi^{a_{5}}$ & \begin{tiny} $I_{\cA}=\{1\}$, \end{tiny} \begin{tiny}$I_{\cL}=\{2\}$, \end{tiny}  & $  q-1$ & $q^1$ \\
\hline
$\mathcal{F}_{\mathbf{6}}$ & $\chi^{a_{6}}$ & \begin{tiny} $I_{\cA}=\{2\}$, \end{tiny} \begin{tiny}$I_{\cL}=\{3\}$, \end{tiny}  & $  q-1$ & $q^1$ \\
\hline
$\mathcal{F}_{\mathbf{7}}$ & $\chi^{a_{7}}$ & \begin{tiny} $I_{\cA}=\{3\}$, \end{tiny} \begin{tiny}$I_{\cL}=\{4\}$, \end{tiny}  & $  q-1$ & $q^1$ \\
\hline
$\mathcal{F}_{\mathbf{8}}$ & $\chi_{b_{2}}^{a_{8}}$ & \begin{tiny} $I_{\cA}=\{1,3\}$, \end{tiny} \begin{tiny}$I_{\cL}=\{5,6\}$, \end{tiny}  & $q^{}(q-1)^{}$ & $q^2$ \\
\hline
$\mathcal{F}_{\mathbf{9}}$ & $\chi_{b_{2}}^{a_{9}}$ & \begin{tiny} $I_{\cA}=\{3\}$, \end{tiny} \begin{tiny}$I_{\cL}=\{6\}$, \end{tiny}  & $q^{}(q-1)^{}$ & $q^1$ \\
\hline
$\mathcal{F}_{\mathbf{10}}$ & $\chi_{b_{3}}^{a_{10}}$ & \begin{tiny} $I_{\cA}=\{2,4\}$, \end{tiny} \begin{tiny}$I_{\cL}=\{6,7\}$, \end{tiny}  & $q^{}(q-1)^{}$ & $q^2$ \\
\hline
$\mathcal{F}_{\mathbf{11}}$ & $\chi_{b_{2}, b_{5}, b_{6}}^{a_{11}}$ & \begin{tiny} $I_{\cA}=\{1,3\}$, \end{tiny} \begin{tiny}$I_{\cL}=\{8,9\}$, \end{tiny}  & $q^{3}(q-1)^{}$ & $q^2$ \\
\hline

\end{longtable}

 \end{center}

\end{document}